\documentclass[10pt,a4paper]{article}
\usepackage[T1]{fontenc}

\usepackage[utf8]{inputenc}
\usepackage{authblk}
\usepackage{amsmath,amssymb,amsthm,mathtools,esint,bm,bbm}
\usepackage{mathrsfs}
\usepackage{bookmark}
\usepackage{amsmath}  

\usepackage{color}
\usepackage{framed}
\usepackage{enumitem}
\allowdisplaybreaks[3]
\newtheorem{definition}{Definition}[section]
\newtheorem{theorem}[definition]{Theorem}
\newtheorem{lemma}[definition]{Lemma}

\newtheorem{corollary}[definition]{Corollary}
\newtheorem{remark}[definition]{Remark}
\numberwithin{equation}{section}
\providecommand{\abs}[1]{\left|#1\right|}
\providecommand{\norm}[1]{\lVert#1\rVert}
\providecommand{\paren}[1]{\left( #1 \right)}
\providecommand{\maren}[1]{\left[ #1 \right]}
\providecommand{\baren}[1]{\left\{ #1 \right\}}
\providecommand{\mean}[1]{\langle #1\rangle}
\providecommand{\e}[1]{\mathbb{E}\left[#1\right]}

\newcommand{\uep}{\ensuremath{u^\varepsilon}}
\newcommand{\uepi}{\ensuremath{u^{\varepsilon_i}}}

\newcommand{\lamep}{\ensuremath{\lambda^\varepsilon}}
\newcommand{\nuep}{\ensuremath{\nu^\varepsilon}}
\newcommand{\nuepi}{\ensuremath{\nu^{\varepsilon_i}}}
\newcommand{\muep}{\ensuremath{\mu^\varepsilon}}
\newcommand{\fep}{\ensuremath{F^\varepsilon}}
\newcommand{\jep}{\ensuremath{J^\varepsilon}}

\newcommand{\vep}{\ensuremath{v^\varepsilon}}
\newcommand{\hep}{\ensuremath{h^\varepsilon}}
\newcommand{\gep}{\ensuremath{g^\varepsilon}}

\renewcommand{\r}{{\mathbb{R}}}
\newcommand{\sd}{{\mathbb{S}^{d-1}}}
\newcommand{\rd}{{\mathbb{R}^d}}

\newcommand{\zd}{{\mathbb{Z}^d}}
\newcommand{\ag}{\ensuremath{L^2(G)}}
\newcommand{\ard}{\ensuremath{L^2(\rd)}}
\newcommand{\prd}{\ensuremath{L^p(\rd)}}
\newcommand{\qrd}{\ensuremath{L^{p'}(\rd)}}
\newcommand{\ardnu}{\ensuremath{L^2(\rd, \nuep)}}
\newcommand{\ardnui}{\ensuremath{L^2(\rd, \nuepi)}}
\newcommand{\ardloc}{\ensuremath{L^2_{\text{loc}}(\rd)}}
\newcommand{\prdloc}{\ensuremath{L^p_{\text{loc}}(\rd)}}
\newcommand{\qrdloc}{\ensuremath{L^{p'}_{\text{loc}}(\rd)}}

\newcommand{\ardw}{\ensuremath{L^2_{w}(\rd)}}

\newcommand{\mathp}{\ensuremath{\mathbb{P}}}

\newcommand{\varezero}{\ensuremath{\varepsilon\rightarrow0}}

\newcommand{\hnorm}{\ensuremath{H^{\alpha/2}(\rd)}}
\newcommand{\wap}{\ensuremath{W^{\alpha/p,p}\left(\mathbb{R}^{d}\right)}}
\newcommand{\hsp}{\hspace*{0.25cm}}

\newcommand{\rmd}{\ensuremath{\,\mathrm{d}}}
\newcommand{\ccinfty}{\ensuremath{C_c^\infty(\rd)}}

\newcommand{\linftyg}{\ensuremath{L^\infty(G)}}
\newcommand{\linftyrd}{\ensuremath{L^\infty(\rd)}}

\newcommand{\lvar}{\ensuremath{L^\varepsilon}}

\newcommand{\intrd}{\ensuremath{\int_{\mathbb{R}^d}}}
\newcommand{\iintrdd}{\ensuremath{\iint_{\mathbb{R}^d\times\mathbb{R}^d}}}
\newcommand{\blv}{\ensuremath{\Lambda^\varepsilon}}

\DeclareMathOperator{\supp}{supp}
\usepackage{hyperref}
\title{Periodic and stochastic homogenization of general nonlocal operators with oscillating
coefficients}

\author[a]{Xiaofeng Jin}
\author[a]{Wentao Huo}
\author[a]{Lingwei Ma}
\author[a]{Zhenqiu Zhang\thanks{Corresponding author.}}

\affil[a]{School of Mathematical Sciences and LPMC, Nankai University, Tianjin, 300071, P.~R.~China}

\date{}
\begin{document}
\maketitle
\footnotetext[1]{E-mail: 1120220040@mail.nankai.edu.cn (X. Jin), 
1120240051@mail.nankai.edu.cn (W. Huo),  malingwei@nankai.edu.cn (L. Ma), 
zqzhang@nankai.edu.cn (Z. Zhang).} 
\begin{abstract}
  This paper investigates homogenization problems for the nonlocal operators with rapidly oscillating coefficients in the cases of periodic and random statistically homogeneous micro-structures. These operators involve the fractional Laplacian and some operators compared to it. Based on  the $\Gamma$-convergence method and compactness arguments, we prove the homogenization theorems for these nonlocal operators with product-type and symmetric coefficient-structured kernels respectively. Furthermore, these results are extended to general nonlinear nonlocal equations.
\end{abstract}

\indent oscillating coefficients; general singular kernels


\section{Introduction}
In this paper, we study the homogenization problems for nonlocal integro-differential equations. The operators under consideration are of the form
\begin{equation}\label{eq,main equation}
\lvar u(x) = \int_{\mathbb{R}^d} (u(y)-u(x))K(x,y) \blv(x,y)\rmd y \quad \text{for } x \in \mathbb{R}^d,
\end{equation}
where the symmetric kernel $K$ is of a very general nature, specifically encompassing both the fractional Laplacian kernel and other kernels comparable to it. The oscillating coefficient $\blv$ is either periodic or  random statistically homogeneous. 

    Fractional and more general nonlocal operators arise naturally in modeling long-range interactions, surpassing the capabilities of classical integer-order differential operators. They play a central role in diverse fields such as anomalous diffusion, image processing, mathematical finance, flow in porous media, and population dynamics \cite{MR1867081, MR2120547, MR2480109, MR3327993}.
To model heterogeneous media with complex microstructures, rapidly oscillating coefficients are introduced into these nonlocal equations \cite{MR4132119}. These coefficients can represent periodic composite materials or statistically homogeneous random media. This wide range of applications provides a strong motivation for the mathematical analysis of such problems, with homogenization theory serving as the key tool to derive effective macroscopic models that encapsulate the large-scale dynamics while incorporating the averaged effect of the small-scale heterogeneities.

    The homogenization theory for second-order elliptic equations in divergence form has a long and well-established history \cite{MR225015, MR240443, MR348255, MR542557, MR712714}. Substantial progress over the past two decades has extended this theory to various directions, including random correctors \cite{MR3713047, MR4103433}, quantitative estimates \cite{MR4377865, MR4840548}, boundary layers \cite{MR3862089, MR3902170}, and multiple homogenization \cite{MR4566686}. Recent work has also studied models with random convolution-type potentials \cite{Jin2025}.We also refer to the monographs \cite{MR3838419} and 
\cite{MR3932093}, which provide comprehensive treatments of large-scale regularity theory and 
convergence rates in periodic and stochastic homogenization, respectively.

    In the nonlocal setting, the study of homogenization theory for nonlocal operators like \eqref{eq,main equation} with integrable kernels  have been extensively developed. For example, in \cite{MR3595876}, Piatnitski and Zhizhina considered the homogenization of the following diffusive scaling nonlocal operators with integrable kernels in periodic media:
\begin{equation}\label{eq:operator with integrable kernel}
  \lvar_1 u(x)=\varepsilon^{-d-2}\intrd \paren{u(y)-u(x)}K\paren{\frac{x-y}{\varepsilon}}\blv(x,y)\rmd y.
\end{equation} 
where the kernel $K(z)\in L^1(\rd)\cap\ardloc$  is nonnegative and symmetric with finite second moments, and the periodic coefficient $\blv(x,y)$ has the product structure $\lambda(x/\varepsilon)\mu(y/\varepsilon)$. Using the corrector method, they proved resolvent convergence to a local second-order elliptic operator. A follow-up work \cite{MR4053030}  extended the initial approach for periodic media to a random framework, in which the coefficient $\blv(x,y)$ takes the form $\lambda(x/\varepsilon,\omega)\mu(y/\varepsilon,\omega)$. Furthermore, under  stronger assumptions on the kernel $K$, quantitative homogenization results can be established, see \cite{MR4532955} and \cite{MR4720001}. Qualitative homogenization results for non-symmetric kernels and for parabolic-type problems also exist \cite{MR4033742, MR4853448}. 

On the other hand,  the kernel of the well-known fractional Laplacian, expressed as $\abs{x}^{-d-\alpha}$ for $\alpha\in(0,2)$, exhibits a clear singularity at the origin. Therefore, nonlocal operators with singular kernels are more general than those with integrable kernels, and have attracted significant interest in the study of homogenization problems.    With $K(x,y)$ introduced at the beginning of the paper chosen as $\abs{x-y}^{-d-\alpha}$, the work \cite{MR4000213} considered the homogenization of the equation
\begin{equation}\label{eq:levy operator}
  \lvar_2 u(x)-mu(x)=\intrd \frac{u(y)-u(x)}{\abs{x-y}^{d+\alpha}}\blv(x,y)\rmd y-mu(x)=f(x).
\end{equation}
 From a formal viewpoint, the operator $\lvar_2$ in \eqref{eq:levy operator} resembles $\lvar_1$ in \eqref{eq:operator with integrable kernel} if we represent $\abs{x-y}^{-d-\alpha}$ as $\varepsilon^{-d-\alpha}K((x-y)/\varepsilon)$. Nevertheless, the lack of integrability of $\abs{z}^{-d-\alpha}$ in $L^1(\rd)$ presents a major obstacle to constructing correctors. As a result, \cite{MR4000213} turned to the $\Gamma$-convergence  and compactness arguments  to homogenize  equation \eqref{eq:levy operator}, in contrast to the corrector method   employed in \cite{MR3595876, MR4053030}.  
 We also mention that \cite{MR4261110} established the almost sure homogenization of symmetric nonlocal Dirichlet forms with the fractional Laplacian kernel $\abs{x-y}^{-d-\alpha}$ and  the one-parameter stationary ergodic coefficient $\Lambda(x/\varepsilon,y/\varepsilon;\omega)$  in
the resolvent sense. Related homogenization results involving kernels of the form $\abs{x-y}^{-d-\alpha}$ can also be found in \cite{MR2560294,MR2981018,MR4348681}.
 
    Despite this progress, existing homogenization theory for singular kernels largely focuses on the fractional Laplacian, leaving general singular kernels relatively unexplored. This paper aims to bridge this gap.  We consider a broad class of symmetric kernels $K(x,y)$ that include the fractional Laplacian as a special case but extend far beyond it, encompassing both convolution-type and more general forms. A key technical challenge in the homogenization analysis of such general operators  is the construction of a suitable decomposition of $\mathbb{R}^d \times \mathbb{R}^d$ that is compatible with our weak assumptions on the kernel $K$. To overcome this, we develop a novel decomposition strategy that generalizes the standard approach used in \cite{MR4000213} for the fractional Laplacian. This new framework is crucial for our analysis and naturally recovers the classical case, underscoring the generality of our results. By extending the $\Gamma$-convergence framework and employing compactness arguments, we prove homogenization theorems in the resolvent sense and further extend the results to a class of nonlinear nonlocal equations.

The remainder of this paper is organized as follows. Section 2 presents the precise assumptions on the kernel $K$ and the coefficient $\blv$, states our main homogenization theorems for linear nonlocal equations, and extends the framework to nonlinear nonlocal equations by introducing relevant assumptions on the nonlinearity and establishing the corresponding homogenization results. Section 3 presents preliminary results and auxiliary lemmas used in the subsequent proofs. Section 4 provides the proofs of Theorems \ref{thm:periodic product structure} and \ref{thm:random product structure} for product-type coefficients via the $\Gamma$-convergence method. Section 5 completes the proofs of Theorems \ref{thm:periodic sym structure} and \ref{thm:random sym structure} for symmetric-type coefficients using a compactness argument. Finally, Section 6 presents the proofs of Theorems \ref{thm1} and \ref{thm:nonlinear random}, concerning the homogenization of nonlinear nonlocal equations. 
\section{Assumptions and main results}
This section presents the precise framework for our homogenization study. We begin by stating the fundamental assumptions on the kernel $K$, which are common to all our results for linear nonlocal operators. We then introduce the two distinct classes of coefficients that will be examined: the product-type and the symmetric-type. For each class, we specify the corresponding additional assumptions (periodic or stationary ergodic) and state the respective homogenization theorem.
\subsection{Assumptions on the kernel}
The kernel $K$ is always assumed to be measurable. The following hypotheses encompass the two primary cases of interest:

\textbf{Case 1:} $K$ is a convolution-type kernel, i.e., $K(x,y)=K(x-y)$.
 
In this case, we assume that $K$ : $\rd\rightarrow\r$ satisfies \eqref{A1}--\eqref{A4} below.
\begin{enumerate}[label=(\textbf{A}\arabic{enumi}),ref=\textbf{A}\arabic{enumi}]
  \item \label{A1} $K$ is nonnegative and symmetric, i.e., for all $z\in\rd$
  \begin{equation*}
  K(z)\geq0\text{,}\quad K(z)=K(-z);
\end{equation*}
  \item \label{A3} There exist constants $\alpha\in(0,2)$ and $C_0>0$
  (depending only on $d$ and $\alpha$) such that for all $r>0$,
  \begin{equation*}
    r^\alpha\int_{\rd\setminus B_r}K(z)\rmd z\leq C_0;
  \end{equation*}
  \item \label{A2}The standard L\'{e}vy-Khintchine condition:
  \begin{equation*}
    \intrd\min\{1,\abs{z}^2\}K(z)\rmd z<\infty;
  \end{equation*}
  \item \label{A4}For the constant $\alpha$ above, there exists a constant $c_0>0$
  (depending only on $d$ and $\alpha$) such that for all $r>0$,
  \begin{equation*}
    r^{\alpha-2}\inf_{e\in\sd}\int_{B_r}\abs{e\cdot z}^2K(z)\rmd z\geq c_0.
  \end{equation*}
\end{enumerate}

\textbf{Case 2:} $K$ is a kernel of non-convolutional type.
 
In this case, we assume that $K$ : $\rd\times\rd\rightarrow\r$ satisfies \eqref{B1}--\eqref{B3} below.
\begin{enumerate}[label=(\textbf{B}\arabic{enumi}),ref=\textbf{B}\arabic{enumi}]
  \item \label{B1} $K$ is nonnegative and symmetric, i.e., for all $x$, $y\in\rd$
  \begin{equation*}
  K(x,y)\geq0\text{,}\quad K(x,y)=K(y,x);
\end{equation*}
  \item \label{B2} There exist constants $\alpha\in(0,2)$ and $C_0>0$
  (depending only on $d$ and $\alpha$) such that for all $r>0$ and $x\in\rd$,
  \begin{equation*}
     r^\alpha\int_{\rd\setminus B_r(x)}K(x,y)\rmd y\leq C_0;
  \end{equation*}
  \item \label{B3}There exist constants $C_1\in(0,1)$ and $C_2>0$ such that for every ball $B\subset\rd$ and $x\in B$,
  \begin{equation*}
    \abs{\{y\in B: K(x,y)\geq C_2\abs{x-y}^{-d-\alpha}\}}\geq C_1\abs{B}.
  \end{equation*}
\end{enumerate}

\begin{remark}
It is straightforward to verify that the kernel of the fractional Laplacian, along with kernels comparable to it, satisfies assumptions \eqref{A1}--\eqref{A4} and \eqref{B1}--\eqref{B3} with appropriate constants.
\end{remark}

\begin{remark}
Assumption \eqref{A4} or \eqref{B3} implies the singularity of the kernel $K$. This singularity can manifest either as non-integrability at the origin or as a blow-up along the diagonal.
Assumption \eqref{B3}, which originates from \cite{MR4111815}, serves an analytical purpose analogous to that of \eqref{A2}--\eqref{A4}, despite their distinct formulations. Specifically, both sets of assumptions are designed to establish a coercivity estimate for the associated bilinear form $(u,u)_K$ (see Lemma \ref{lemma:equivalent norm}), but they apply to structurally different classes of kernels: \eqref{A2}--\eqref{A4} are tailored for convolution-type kernels $K(x-y)$, whereas \eqref{B3} applies to general symmetric kernels $K(x,y)$.
\end{remark}

The kernel assumptions adopted in this work are well-established in the theory of integro-differential equations. They appear, for instance, in the analysis of kinetic equations such as the non-cutoff Boltzmann equation \cite{MR4049224}. This structural compatibility provides a key motivation for adopting \eqref{B1} and \eqref{B2} in our homogenization study. Related assumptions \eqref{A1}--\eqref{A4} have also been employed in the analysis of nonlocal operators \cite{MR4793281}. For a comprehensive probabilistic motivation of assumptions \eqref{A1}--\eqref{A4} in the context of stochastic processes, we refer to the recent monograph \cite{MR4769823}.

\subsection{Homogenization results for product-type coefficients}

We first consider the periodic product structure. 
 We say that $h$ is a $1$-periodic function if $h(x+z)=h(x)$ for all $x\in\rd$ and $z\in\zd$.
 We denote by $\langle h\rangle$ the mean value of $1$-periodic function  $h$, namely, 
 \begin{equation*}
  \mean h:= \int_{[0,1]^d}h(x)\rmd x.
 \end{equation*} 
In this setting we assume that the  coefficient $\blv$ has the following form:
\begin{equation}\label{eq:perodic product}
  \blv(x,y)=\lamep(x)\muep(y):= \lambda\paren{x/\varepsilon}\mu\paren{y/\varepsilon},
\end{equation}
where both $\lambda$ and $\mu$ satisfy periodicity:
\begin{equation}\label{eq:lambda and mu 1-periodic}
  \lambda(x+z)=\lambda(x)\text{,} \hsp\mu(x+z)=\mu(x)
\end{equation}
for all $x\in\rd$ and $z\in\zd$. In addition, they satisfy the bounds
 \begin{equation}\label{eq:lambda mu bound}
   1/\gamma\leq\lambda(x)\leq\gamma\text{,}\hsp 1/\gamma\leq\mu(x)\leq\gamma
 \end{equation} 
for some constant $\gamma\geq1$ and all $x\in\rd$.

Due to the lack of regularities for the coefficient $\blv$ and the kernel $K$, the expression $\lvar u(x)$  may not be well defined in the classical sense even for $u\in \ccinfty$. Therefore, throughout this paper, we work with weak solutions.  We now provide the corresponding definition for periodic $\blv$ defined in \eqref{eq:perodic product} and the definition for random $\blv$ defined in \eqref{eq:random product} below can be formulated analogously.
\begin{definition}
  For  $m>0$ and  $f\in\ard$, we say that $\uep\in\hnorm$ is a (weak) solution of
\begin{equation}\label{eq:uep satisfies equation}
    \paren{m-\lvar}\uep=f,
  \end{equation}
if, for all $v\in\hnorm$, 
\begin{equation}\label{eq:definition of weak solution of lvar product form}
\begin{aligned}
   m\intrd \uep(x)v(x)\nuep(x)\rmd x&-\iintrdd \paren{\uep(y)-\uep(x)}K(x,y)v(x)\muep(x)\muep(y)\rmd y\rmd x\\
 &=\intrd f(x)v(x)\nuep(x)\rmd x,
 \end{aligned}
\end{equation} 
where $\nuep:=\muep/\lamep$.
Formally, this can be written as
\begin{equation}\label{eq:weak solution def1}
  m(\uep,v)_{\ardnu}-\paren{\lvar \uep, v}_{\ardnu}=(f,v)_{\ardnu}.
\end{equation}
\end{definition} 

To illustrate the main results of this paper, for $\varepsilon\rightarrow 0$, we first present the definition of the limit operator $L^0$:
\begin{equation}\label{eq:L0 periodic product structure}
      L^0u(x)=\intrd \overline{\Lambda}(x,y)\paren{u(y)-u(x)}K(x,y)\rmd y,
    \end{equation}
    where 
    \begin{equation}\label{eq:lambda bar for periodic product}
      \overline{\Lambda}(x,y)=\frac{\mean{\mu}^2}{\mean{\mu/\lambda}}.
    \end{equation}
    Similarly, we define a weak solution for $\paren{m-L^0}u^0=f$ as follows.
    \begin{definition}
      For  $m>0$ and  $f\in\ard$,  we say that $u^0\in\hnorm$ is a (weak) solution of  
\begin{equation}\label{eq:u0 satisfies equation}
    \paren{m-L^0}u^0=f,
    \end{equation}
if, for all $v\in\hnorm$, 
\begin{equation}\label{eq:definition of weak solution L0 product form}
   m\intrd u^0(x)v(x)\rmd x-\frac{\mean{\mu}^2}{\mean{\mu/\lambda}}\iintrdd \paren{u^0(y)-u^0(x)}K(x,y)v(x)\rmd y\rmd x
 =\intrd f(x)v(x)\rmd x.
\end{equation}  
 Formally, this can be written as
\begin{equation*}
  m(u^0,v)-\paren{L^0 u^0, v}=(f,v).
\end{equation*}  
Here and subsequently, for simplicity of notation, we denote $(\cdot,\cdot):=(\cdot,\cdot)_{L^{2}\left(\mathbb{R}^{d}\right)}$.

    \end{definition}
    
The existence and uniqueness of the limit weak solution $u^0$ and the solution $\uep$ will be proved in Section 3  via variational techniques.
The first homogenization result is presented below. Note that we adopt uniform notation $K(x,y)$ for both kernel functions considered in this paper when presenting our results unless otherwise stated.
\begin{theorem}\label{thm:periodic product structure}
  Assume that\eqref{A1}--\eqref{A4} or \eqref{B1}--\eqref{B3} hold. Let $\blv$ satisfy \eqref{eq:perodic product}--\eqref{eq:lambda mu bound}. For a given constant $m>0$ and each $f\in\ard$, let $\uep$ be the solution of \eqref{eq:uep satisfies equation} and $u^0$ be the solution of
  \eqref{eq:u0 satisfies equation}.
 Then, as $\varezero$, $\uep$ converges to $u^0$ strongly in $\ard$ and weakly in  $\hnorm$.
\end{theorem}
We prove Theorem \ref{thm:periodic product structure} via the $\Gamma$-convergence method, 
which was first proposed by De Giorgi in the 1970s (see \cite{MR375037} and \cite{MR448194}) and it has  become an essential tool for studying the convergence of variational problems. For a comprehensive treatment of this convergence and its properties, we refer the reader to Dal Maso's monograph \cite{MR1201152}. Additionally, Braides' book \cite{MR1968440} provides numerous applications of this convergence framework across various contexts.

Next we study the homogenization problem for \eqref{eq,main equation} in random environment. Firstly we need to make some assumptions concerning the standard probability space $\Omega$. We assume that $(\Omega,\mathcal{F},\mathbb{P})$ 
is equipped with an $n$-dimensional measure preserving and ergodic
dynamical system $\{\tau_x\}_{x\in \rd}$. More precisely, 
for every $x, y\in \rd$ and $F\in \mathcal{F}$, 
\begin{itemize}
  \item  Group property: $\tau_x \circ \tau_y =\tau_{x+y}$, $\tau_0 =\bf{Id}$,
  where $\bf{Id}$ is the identical map;
  \item Measure preservation: $\tau_xF \in \mathcal{F}$ and 
  $\mathbb{P}\left(F\right)=\mathbb{P}\left(\tau_xF\right)$;
  \item Measurability: the function $(x, \omega)\to \tau_x\omega$ is
   $\mathcal{B}(\rd)\otimes \mathcal{F}$-measurable, where $\mathcal{B}$ is the 
   Borel $\sigma$-algebra; 
  \item Ergodicity: $\mathbb{P}[(\tau_xF)\triangle(F)]=0$ for all $x\in \rd$ if and only if 
  $\mathp(F)=0$ or $\mathp(F)=1$, where $(\tau_xF)\triangle(F)$ is the symmetric difference.
\end{itemize} 
In this case we assume that $\blv$ has the following form:
\begin{equation}\label{eq:random product}
  \blv(x,y)=\lamep(x)\muep(y):=\lambda\paren{x/\varepsilon,\omega}\mu\paren{y/\varepsilon,\omega},
\end{equation}
where both $\lambda$ and $\mu$ are stationary in the sense that there exist random variables $\boldsymbol{\lambda}$ and $\boldsymbol{\mu}$ such that almost surely, for all $x\in\rd$,
\begin{equation}\label{eq:lambda mu stationary}
  \lambda(x,\omega)=\boldsymbol{\lambda}(\tau_x\omega) \quad\text{and}\quad \mu(x,\omega)=
  \boldsymbol{\mu}(\tau_x\omega).
\end{equation} 
 Similar to the periodic case \eqref{eq:lambda mu bound}, we also assume that there exists a constant $\gamma\geq1$ such that, almost surely, 
 \begin{equation}\label{eq:lambda mu random bound}
   1/\gamma\leq\boldsymbol\lambda\leq\gamma\text{,}\hsp 1/\gamma\leq\boldsymbol\mu\leq\gamma.
 \end{equation} 
 The theorem below is analogue of Theorem \ref{thm:periodic product structure}. It shows that the operator $\lvar$ still admits homogenization for random $\blv$ with product structures. 
\begin{theorem}\label{thm:random product structure}
 Under the assumptions of Theorem \ref{thm:periodic product structure} with \eqref{eq:perodic product}--\eqref{eq:lambda mu bound} replaced by \eqref{eq:random product}--\eqref{eq:lambda mu random bound} and with \eqref{eq:lambda bar for periodic product} replaced by
   \begin{equation}\label{eq:lambda bar for random product}
    \overline{\Lambda}(x,y) =\frac{\e{\boldsymbol\mu}^2}{\e{\boldsymbol\mu/\boldsymbol\lambda}},
    \end{equation}
    the conclusion of Theorem \ref{thm:periodic product structure} still holds in the almost sure sense.
\end{theorem}
\subsection{Homogenization results for symmetric-type coefficients}
Next We consider the case where the coefficient $\blv$ is of periodic symmetric structure. Precisely, we assume
\begin{equation}\label{eq:s1}
	\Lambda^{\varepsilon}(x,y)=\Lambda\left(x,y,x/\varepsilon,y/\varepsilon\right) 
\end{equation} 
with a function $\Lambda(x,y,\xi,\eta)$ that is continuous in $(x,y)$, measurable in $(\xi,\eta)$ for all $(x,y)$, and satisfies periodicity:
\begin{equation}\label{eq:s2}
	\Lambda(x,y,\xi+z,\eta)=\Lambda(x,y,\xi,\eta), \;\;  \Lambda(x,y,\xi,\eta+z)=\Lambda(x,y,\xi,\eta)
\end{equation}
for all $x, y,\xi, \eta\in\rd$ and $z\in\zd$. Moreover, $\Lambda(x,y,\xi,\eta)$ satisfies the following conditions:
\begin{align}\label{eq:s3}
	\begin{cases}
		\Lambda(x,y,\xi,\eta)=\Lambda(x,y,\eta,\xi),\\
		\gamma^{-1}\leq \Lambda(x,y,\xi,\eta)\leq \gamma,
	\end{cases}
\end{align}
for some constant $\gamma\geq1$ and all $x,y,\xi,\eta\in \mathbb{R}^{d}$. 

We first provide the corresponding definition of  the weak solution of \eqref{eq:uep satisfies equation}  for periodic $\blv$ defined in \eqref{eq:s1} and the definition for random $\blv$ defined in \eqref{form} below can be formulated analogously. 
\begin{definition}
  For  $m>0$ and  $f\in\ard$, we say that $\uep\in\hnorm$ is a (weak) solution of \eqref{eq:uep satisfies equation},
if, for all $v\in\hnorm$, 
\begin{equation*}
	m\intrd u^{\varepsilon}(x)v(x)\rmd x-\iintrdd \paren{u^{\varepsilon}(y)-u^{\varepsilon}(x)}K(x,y)\Lambda^{\varepsilon}(x,y)v(x)\rmd y\rmd x
	=\intrd f(x)v(x)\rmd x.
\end{equation*} 
Formally, this can be written as
\begin{equation}\label{eq:weak solution def2}
	m\left(u^{\varepsilon},v\right)-\paren{\lvar u^{\varepsilon}, v}=(f,v).
\end{equation}
\end{definition}

We proceed to present the definition of the limit operator $L^0$:
\begin{equation}\label{eq:L0 for random sym structure}
	L^0u(x)=\intrd \overline{\Lambda}(x,y)\paren{u(y)-u(x)}K(x,y)\rmd y,
\end{equation}
where
\begin{equation}\label{coe}
	\overline{\Lambda}(x,y)=\iint_{[0,1]^{d}\times[0,1]^{d}}\Lambda(x,y,\xi,\eta)\rmd\xi \rmd \eta.
\end{equation}
\begin{definition}
  For  $m>0$ and  $f\in\ard$,  we say that $u^0\in\hnorm$ is a (weak) solution of  \eqref{eq:u0 satisfies equation}, if, for all $v\in\hnorm$, 
\begin{equation*}
	m\intrd u^0(x)v(x)\rmd x-\iintrdd \paren{u^{0}(y)-u^{0}(x)}K(x,y)\overline{\Lambda}(x,y)v(x)\rmd y\rmd x
	=\intrd f(x)v(x)\rmd x.
\end{equation*}  
Formally, this can be written as
\begin{equation*}
	m(u^0,v)-\paren{L^0 u^0, v}=(f,v).
\end{equation*}  
\end{definition}

We are now in a position to state the homogenization result for the periodic symmetric coefficient $\blv$. 
\begin{theorem}\label{thm:periodic sym structure}
	Assume that \eqref{A1}--\eqref{A4} or \eqref{B1}--\eqref{B3} hold. Let $\blv$ satisfy \eqref{eq:s1}--\eqref{eq:s3}. For a given constant $m>0$ and each $f\in\ard$, let $\uep$ be the solution of \eqref{eq:uep satisfies equation}
	and $u^0$ be the solution of \eqref{eq:u0 satisfies equation} with the limit operator $L^0$ defined in \eqref{eq:L0 for random sym structure}. Then, as $\varepsilon\rightarrow 0$, $\uep$ converges to $u^0$ strongly in $\ard$ and weakly in  $\hnorm$  almost surely.
\end{theorem}

We proceed to study the homogenization problem for \eqref{eq,main equation} in random symmetric environment. Beyond the assumptions on the probablity space $(\Omega,\mathcal{F},\mathbb{P})$ made in the random product setting, we additionally assume some topological
structure. We assume that $\Omega$ is a metric compact space, that $\mathcal{F}$ is the
Borel $\sigma$-algebra of $\Omega$, and that the group $\tau_{\cdot}$ is continuous. Moreover, we assume that $\blv$ has the following form:
\begin{equation}\label{form}
	\blv(x,y)=\Lambda\left(x,y,\tau_{\frac{x}{\varepsilon}}\omega,\tau_{\frac{y}{\varepsilon}}\omega\right),
\end{equation} 
where $\Lambda\left(x,y,\omega_{1},\omega_{2}\right)$ is a continuous function on $\rd\times\rd\times\Omega\times\Omega$ and satisfies the
following symmetry condition:
\begin{equation}\label{symm}
	\Lambda\left(x,y,\omega_{1},\omega_{2}\right)=\Lambda\left(y,x,\omega_{2},\omega_{1}\right).
\end{equation} 
In addition, $\Lambda$ is  uniformly positive and bounded in the sense that almost surely, for all $x,y\in\rd$,
\begin{equation}\label{stationary}
	 1/\gamma\leq{\Lambda}\leq\gamma.
\end{equation} 

The following result shows that the operator $\lvar$ still admits homogenization for random $\blv$ with symmetric structure. 
\begin{theorem}\label{thm:random sym structure}
	Under the assumptions of Theorem \ref{thm:periodic sym structure} with \eqref{eq:s1}--\eqref{eq:s3} replaced by \eqref{form}--\eqref{stationary} and with \eqref{coe} replaced by
	\begin{equation}\label{eq:lambda bar}
		\overline{\Lambda}(x,y)=\iint_{\Omega\times\Omega}\Lambda(x,y,\omega_{1},\omega_{2})
\rmd{\mathbb{P}}(\omega_{1})\rmd{\mathbb{P}}(\omega_{2}),
	\end{equation}
	the conclusion of Theorem \ref{thm:periodic sym structure} still holds.
\end{theorem}
Note that we will prove Theorems \ref{thm:periodic sym structure} and \ref{thm:random sym structure} via a compactness argument, which differs from the $\Gamma$-convergence method employed for Theorems \ref{thm:periodic product structure} and \ref{thm:random product structure}. Nevertheless, these two methods are  compatible, and each can be adapted to derive the results proved by the other.

\subsection{Extension to nonlinear nonlocal equations}
Finally, the homogenization framework developed in the previous section can be extended to a general class of nonlinear nonlocal equations. This extension demonstrates the robustness of our approach. In this section, we state the nonlinear analogues of our main results. Specifically, motivated by \cite{MR3339179, MR4693935}, we study the following nonlocal operator:
\begin{equation}\label{eq:nonlinear operator}
 	L^{\varepsilon}_{\Phi}u(x)=\intrd K(x,y)\Phi'\left(u(y)-u(x)\right)\Lambda^{\varepsilon}(x,y)\rmd y,\quad x\in\mathbb{R}^{d}.
 \end{equation}
 The kernel $K$ : $\mathbb{R}^{d}\times\mathbb{R}^{d}\rightarrow\r$ is assumed to be measurable and symmetric,  satisfying
 \begin{equation}\label{1}
 	\frac{1}{\Upsilon_1|x-y|^{d+\alpha}}\leq K(x,y)\leq \frac{\Upsilon_1}{|x-y|^{d+\alpha}}
 \end{equation}
  for all $x$, $y\in\rd$, $x\neq y$, where $\Upsilon_1\geq 1$ and $\alpha\in(0,p)$ with $p>1$. 
  The function $\Phi$ : $\r\to\r$ is assumed to be strictly convex and continuously differentiable with the odd derivation function $\Phi'$, satisfying
  \begin{equation}\label{eq:Phi condition}
    \Upsilon_2^{-1}\abs{t}^p\leq\Phi(t)\leq\Upsilon_2\abs{t}^p,\quad
    \Phi'(t)t\geq\Upsilon_{2}^{-1}\abs{t}^p,
  \end{equation}
  and
  \begin{equation}\label{eq:Phi derivation condition}
  \abs{\Phi'(s)-\Phi'(t)}\leq 
	\begin{cases}
		C(p)|s-t|^{p-1} &\quad\text{if}\hsp 1<p\leq2, \\
		C(p)\left(|s|+|t|\right)^{p-2}|s-t| &\quad \text{if}\hsp p>2
	\end{cases}
   \end{equation}
  for all $s$, $t\in\r$ and some fixed constants  $\Upsilon_2\geq 1$, $C(p)> 0$.
   Moreover, we assume that  
  $\Lambda^{\varepsilon}(x,y)$ satisfies the assumptions on it stated in any one of Theorems \ref{thm:periodic product structure}--\ref{thm:random product structure} and Theorems \ref{thm:periodic sym structure}--\ref{thm:random sym structure}. 
\begin{remark}
  It is straightforward to verify that the function $\Phi(t)=\abs{t}^{p}$ satisfies all the assumptions imposed on $\Phi$ above. For the verification of assumption \eqref{eq:Phi derivation condition}, one may refer to \cite[pp.98]{MR3931688} for $1<p\leq 2$, and \cite[Lemma 2.2]{MR847306}(see also \cite[Lemma 2.2]{MR4919132}) for $p>2$.
\end{remark}

 For the nonlinear nonlocal equation
 \begin{equation}\label{52}
 	-L^{\varepsilon}_{\Phi}u^{\varepsilon}+m\abs{\uep}^{p-2}\uep=f,
 \end{equation}
 where $m>0$, $p>1$, and $f\in L^{p'}\left(\mathbb{R}^{d}\right)$ with $p'=\frac{p}{p-1}$, we adopt  similar definitions of weak solutions as in the linear case (see \eqref{eq:weak solution def1} and \eqref{eq:weak solution def2} for $\blv(x,y)$ with product and symmetric structure, respectively), with the linear operator $\lvar$ replaced by the nonlinear operator $-L^{\varepsilon}_{\Phi}u^{\varepsilon}$ and the linear lower-order term $\uep$
 replaced by the nonlinear lower-order term $\abs{\uep}^{p-2}\uep$.

 We proceed to introduce the limit operator
\begin{equation*}
	L_{\Phi}^{0}u(x)=\int_{\mathbb R^d} K(x,y)\Phi'\left(u(y)-u(x)\right)\overline{\Lambda}(x,y)\rmd y,\quad x\in\mathbb{R}^{d},
\end{equation*}
and the limit equation
\begin{equation}\label{54}
	-L^{0}_{\Phi}u^0+m\abs{u^0}^{p-2}u^0=f.
\end{equation}

The following theorem presents a homogenization result for the nonlinear nonlocal problem \eqref{52} with periodic coefficient $\blv(x,y)$.
 \begin{theorem}[Periodic coefficients]\label{thm1}
 	 Let $m>0$, $p>1$, $\alpha\in(0,p)$, $p'=\frac{p}{p-1}$, and $f\in L^{p'}\left(\mathbb{R}^{d}\right)$. Assume that $u^{\varepsilon}$ be a weak solution of \eqref{52} with coefficient $\blv(x,y)$ defined in \eqref{eq:perodic product}--\eqref{eq:lambda mu bound} or \eqref{eq:s1}--\eqref{eq:s3},
 	and that $u^{0}$ be a weak solution of \eqref{54} with the corresponding homogenized coefficient $\overline{\Lambda}(x,y)$ defined in \eqref{eq:lambda bar for periodic product} or \eqref{coe}.
 	Then, as $\varepsilon\rightarrow0$, $u^{\varepsilon}$ converges to $u^{0}$ strongly in $L^{p}\left(\mathbb{R}^{d}\right)$ and weakly in  $W^{\alpha/p,p}\left(\mathbb{R}^{d}\right)$. 
 \end{theorem}
 Analogously, for the nonlinear nonlocal problem \eqref{52} with random coefficient $\blv(x,y)$, the corresponding homogenization result holds as stated below.
 \begin{theorem}[Random coefficients]\label{thm:nonlinear random}
   Under the assumptions of Theorem \ref{thm1}, let $u^{\varepsilon}$ be a weak solution of \eqref{52} with coefficient $\blv(x,y)$ defined in \eqref{eq:random product}--\eqref{eq:lambda mu random bound} or \eqref{form}--\eqref{stationary},
 	and let $u^{0}$ be a weak solution of \eqref{54} with the corresponding homogenized coefficient $\overline{\Lambda}(x,y)$ defined in \eqref{eq:lambda bar for random product} or \eqref{eq:lambda bar}.
 	Then, the conclusion of Theorem \ref{thm1} still holds in the almost sure sense. 
 \end{theorem}

\section{Preliminaries and auxiliary results}
We begin by recalling the definition of the fractional Sobolev spaces, which will serve as the natural functional setting for our problem. The fractional Sobolev space $W^{s,p}(\rd)$ with $p\in[1,\infty)$ and $s\in(0,1)$ is defined as
\begin{equation*}
  W^{s,p}(\rd):= \baren{u\in L^p(\rd) : \norm{u}_{W^{s,p}(\rd)}<\infty},
\end{equation*}
with norm
\begin{equation*}
  \norm{u}_{W^{s,p}(\rd)}=\norm{u}_{L^p(\rd)}+[u]_{W^{s,p}(\rd)},
\end{equation*}
and the Gagliardo seminorm
\begin{equation*}
  [u]_{W^{s,p}(\rd)} := \paren{\iintrdd\frac{\paren{u(x)-u(y)}^p}{\abs{x-y}^{d+sp}}\rmd x\rmd y}^{1/p}.
\end{equation*}
In the particular case $p=2$, $W^{s,2}(\rd)$ becomes a Hilbert space, commonly denoted by 
$H^s(\rd)$. In what follows, the notation $a\asymp b$ means that $a$ and $b$ are comparable functions
in the sense that there exists a constant $C\geq1$ independent of $a$ and $b$ such that
$a/C\leq b\leq Ca$ in a given domain.

The following lemma shows that the previous assumptions regarding the kernel function $K$ ensure that the nonlocal operator $\lvar$ in \eqref{eq,main equation} satisfies boundedness and coercivity estimates analogous to those of fractional Laplace type operators.
 \begin{lemma}\label{lemma:equivalent norm}
  Assume \eqref{A1}--\eqref{A4} or \eqref{B1}--\eqref{B3}. Let $\blv(x,y)$ satisfy
  $\blv(x,y)=\blv(y,x)$ and $1/C_3\leq\blv(x,y)\leq C_3$ with a constant $C_3\geq1$ for all $x$, $y\in\rd$. Then for all $u\in\hnorm$, we have 
  \begin{equation}\label{eq:equi norm sym form}
  \paren{-\lvar u, u}_{\ard}=\iintrdd \paren{u(x)-u(y)}K(x,y)\blv(x,y)u(x)\rmd y\rmd x
  \asymp[u]_{\hnorm}^2.  
  \end{equation}
  
 \end{lemma}
 \begin{proof}
   By the symmetry of $K$ and $\blv$, exchanging the integration variables shows that
   \begin{equation*}
     \begin{aligned}
     &\iintrdd \paren{u(x)-u(y)}K(x,y)\blv(x,y)u(x)\rmd y\rmd x\\
     &=\iintrdd \paren{u(y)-u(x)}K(x,y)\blv(x,y)u(y)\rmd x\rmd y\\
     &=\frac{1}{2}\iintrdd \paren{u(x)-u(y)}^2K(x,y)\blv(x,y)\rmd x\rmd y\\
     &\asymp\iintrdd \paren{u(x)-u(y)}^2K(x,y)\rmd x\rmd y.
     \end{aligned}
   \end{equation*} 
   We are reduced to proving
   \begin{equation}\label{eq:equai norm}
    (u,u)_K := \iintrdd \paren{u(x)-u(y)}^2K(x,y)\rmd x\rmd y\asymp[u]_{\hnorm}^2.
   \end{equation}
    In  Case 1, the proof of \eqref{eq:equai norm} can be found in \cite[Lemma 2.2.17]{MR4769823}. While in Case 2,  we refer the reader to \cite[Theorem 1.2]{MR4111815} and \cite[Lemma 4.2]{MR4049224}.
 \end{proof}
 \begin{remark}\label{remark:equi norm product form}
    The same proof of Lemma \ref{lemma:equivalent norm} remains valid for the product-structured $\blv$, i.e., \eqref{eq:perodic product} or \eqref{eq:random product}, with a slight change of \eqref{eq:equi norm sym form}:
   \begin{equation*}
     \paren{-\lvar u, u}_{\ardnu} =\iintrdd \paren{u(x)-u(y)}K(x,y)u(x)\muep(x)\muep(y)\rmd y\rmd x
  \asymp[u]_{\hnorm}^2.
   \end{equation*}
   \end{remark}
 Hereinafter, we use $C$ to denote various positive constants, possibly different from line to line, whose dependence on other parameters will be clear from the context unless explicitly stated. We  also write $a\lesssim b$ instead of $a\leq Cb$ for uniform constant $C$ which does not need to be represented exactly and may change from line to line.  
   
  The following a prior estimate is an immediate consequence of Lemma \ref{lemma:equivalent norm} and Remark \ref{remark:equi norm product form} if we take $\uep$ itself as a test function.
  \begin{corollary}\label{corollary:a prior estimate}
    For the solutions $\uep$ in Theorems \ref{thm:periodic product structure}--\ref{thm:random product structure} and Theorems \ref{thm:periodic sym structure}--\ref{thm:random sym structure}, we have
    \begin{equation}\label{eq:a prior estimate}
     \norm{\uep}_{\hnorm}\lesssim\norm{f}_{\ard}.
    \end{equation}
  \end{corollary}
 
 Throughout the paper we will frequently use the fact that $h(x/\varepsilon)\rightarrow\mean h$ 
 weakly in $\prdloc$ as $\varezero$ for $p\in[1,\infty)$
 whenever $h$ is a $1$-periodic function and $h\in L^p([0,1]^d)$. 
 For the proof of this fact, we refer the reader to \cite[Proposition 2.2.2]{MR3838419}
 or \cite[Section 1.1]{jikov2012homogenization}.
 
 We now state some auxiliary lemmas and remarks concerning convergence that will be instrumental in proving our main results.
 \begin{lemma}\label{lemma:periodic weak convergence}
   Let $h\in\linftyrd$ be a $1$-periodic function. Fix $p\in(1,\infty)$, and let $g^\varepsilon$ be bounded
uniformly in $\prd$ for all $0<\varepsilon<1$. Assume that there exists a function 
$g\in\prd$ such that $g^\varepsilon\rightarrow g$ strongly in $\prdloc$ as $\varezero$.
 Then, $h(\cdot/\varepsilon)g^\varepsilon\rightarrow\langle h\rangle g$ weakly in $\prd$
as $\varepsilon\rightarrow0$.
 \end{lemma}
 \begin{proof}
 Let $h^\varepsilon(x) := h(x/\varepsilon)$.
   It is sufficient to prove that  
   \begin{equation}\label{eq:hg weak}
    \lim_{\varezero} \intrd h^\varepsilon g^\varepsilon\varphi=
     \intrd\langle h\rangle g\varphi
   \end{equation}
   for all $\varphi\in\qrd$, where $p'$ is the conjugate exponent of $p$.
   We first assume that $\varphi\in\ccinfty$. By the triangle inequality and H\"older's 
   inequality, we have
   \begin{equation*}
     \begin{aligned}
   \abs{\intrd\left(h^\varepsilon g^\varepsilon\varphi-\langle h\rangle g\varphi\right)}
  &\leq\abs{\intrd\paren{h^\varepsilon g^\varepsilon\varphi-
   h^\varepsilon g\varphi}}+
   \abs{\intrd\paren{h^\varepsilon g\varphi-\langle h\rangle g\varphi}}\\
  &\leq\paren{\int_{\supp \varphi}\abs{g^\varepsilon-g}^p}^{1/p}
  \paren{\int_{\supp \varphi}\abs{h^\varepsilon\varphi}^q}^{1/q}+
  \abs{\int_{\supp \varphi}\paren{h^\varepsilon-\langle h\rangle}g\varphi}.
  \end{aligned}
   \end{equation*}
  Observe that  $h^\varepsilon\varphi$ is bounded uniformly in $\linftyrd$ and that $g\varphi\in\prd$.
  Hence we conclude from the fact that $g^\varepsilon\rightarrow g$  strongly in $\prdloc$ as $\varezero$
  and that $h^\varepsilon\rightarrow\mean{h}$ weakly in $\qrdloc$ as $\varezero$ that
  \begin{equation*}
    \limsup_{\varezero}\abs{\intrd\paren{h^\varepsilon g^\varepsilon-\mean h g}\varphi}
    =0.
  \end{equation*}
  We now turn to the case when $\varphi\in\qrd$. For each $\delta>0$, we  choose $\varphi^\delta\in\ccinfty$ such that $\norm{\varphi^\delta-\varphi}_{\qrd}<\delta$. Therefore, applying the triangle inequality and H\"older's 
   inequality to derive
  \begin{equation*}
     \begin{aligned}
   &\limsup_{\varezero}\abs{\intrd\left(h^\varepsilon g^\varepsilon\varphi-\langle h\rangle g\varphi\right)}\\
  &\leq\limsup_{\varezero}\abs{\intrd\paren{h^\varepsilon g^\varepsilon\paren{\varphi-\varphi^\delta}}}+
   \limsup_{\varezero}\abs{\intrd\paren{h^\varepsilon g^\varepsilon-\mean h g}\varphi^\delta}
   +\limsup_{\varezero}\abs{\intrd\mean h g\paren{\varphi-\varphi^\delta}}\\
  &\leq\limsup_{\varezero}\norm{h}_{\linftyrd}\norm{g^\varepsilon}_{\prd}\norm{\varphi-\varphi^\delta}_{\qrd}
  +0+\abs{\mean h}\norm{g}_{\prd}\norm{\varphi-\varphi^\delta}_{\qrd}\\
  &\leq C\delta,
  \end{aligned}
   \end{equation*}
   where the constant $C>0$ is independent of $\delta$ in the last inequality. By the arbitrariness  of $\delta>0$, letting $\delta\rightarrow0$ yields the desired formula \eqref{eq:hg weak}.
 \end{proof}
 \begin{remark}\label{remark:random weak convergence}
     If we replace the assumption that $h\in\linftyrd$ is a $1$-periodic function with the assumption that almost surely $h(x,\omega)\in\linftyrd$  is stationary in the sense of \eqref{eq:lambda mu stationary}, i.e., there exists a random variable $\boldsymbol{h}$ such that almost surely $h(x,\omega)=\boldsymbol{h}(\tau_x\omega)$  for all $x\in\rd$,
     we can apply a proof that is nearly identical to the one for Lemma \ref{lemma:periodic weak convergence} to show that
 $h(\cdot/\varepsilon,\omega)g^\varepsilon\rightarrow\e{\boldsymbol{h}}g$  weakly in $\prd$ almost surely as $\varezero$. But here, instead of using the fact that  $h^\varepsilon\rightarrow\mean{h}$ weakly in $\qrdloc$, we need to employ that $\hep(x,\omega):=h(x/\varepsilon,\omega)\to\e{\boldsymbol{h}}$ weakly in $\qrdloc$ almost surely, which is a standard consequence of the Birkhoff ergodic theorem \cite[Proposition 7.2]{jikov2012homogenization}.
 \end{remark}
 \begin{lemma}\label{lemma:bounde domain periodic convergence}
   Let $G$ be a bounded measurable subset in $\rd\times\rd$, and let $k(x,y)\in\linftyg$. Then, under the assumptions of Lemma \ref{lemma:periodic weak convergence} for $p=2$, we have
    \begin{equation}\label{eq:bounde domain periodic convergence}
      \lim_{\varezero}\iint_{G}\hep(x)\hep(y)\paren{\gep(x)-\gep(y)}^2k(x,y)\rmd x\rmd y=
     \mean{h}^2\iint_{G}\paren{g(x)-g(y)}^2k(x,y)\rmd x\rmd y.
    \end{equation}
    \end{lemma}
    \begin{proof}
      We begin by proving that $\hep(x)\hep(y)\rightarrow\mean{h}^2$ weakly in $\ag$ as $\varezero$. Although a simpler proof is available for the case of periodic $h$, we present a unified argument that also applies to the case of random $h$ discussed in Remark \ref{remark:bounded domain random convergence} below.  We first assume that $G$ is a closed rectangle in $\rd\times\rd$, i.e., $G=D\times E$ for some closed rectangles $D$, $E\subset\rd$. Since $L^2(D)$ and $L^2(E)$ are separable Hilbert spaces, for any $\varphi(x,y)\in L^2(D\times E)$ and $\theta_1>0$,  we can choose a finite set of product functions 
      \begin{equation*}
        \{\varphi_1^i(x)\varphi_2^i(y):\varphi_1^i\in L^2(D), \varphi_2^i\in L^2(E),  i=1,2,\ldots,N(\varphi,\theta_1)\},
      \end{equation*}
      such that
      \begin{equation*}
        \norm{\varphi(x,y)-\sum_{i=1}^{N}\varphi_1^i(x)\varphi_2^i(y)}_{L^2(D\times E)}\leq\theta_1.
      \end{equation*}
     Observe that $\abs{\hep}\leq C$ almost everywhere and that $\abs{D\times E}<\infty$. Hence, it follows from the triangle inequality and H\"older's inequality
      that
      \begin{align*}
        & \limsup_{\varezero}\abs{\iint_{D\times E}\hep(x)\hep(y)\varphi(x,y)\rmd x\rmd y-\iint_{D\times E}\mean{h}^2\varphi(x,y)\rmd x\rmd y}\\
        &\leq \limsup_{\varezero}\abs{\iint_{D\times E}\hep(x)\hep(y)\paren{\varphi(x,y)-\sum_{i=1}^{N}\varphi_1^i(x)\varphi_2^i(y)}\rmd x\rmd y}\\
       &\quad +\limsup_{\varezero}\sum_{i=1}^{N}\abs{\iint_{D\times E}\hep(x)\hep(y)\varphi_1^i(x)\varphi_2^i(y)\rmd x\rmd y-\iint_{D\times E}\mean{h}^2\varphi_1^i(x)\varphi_2^i(y)\rmd x\rmd y}\\
       &\quad+\limsup_{\varezero} \abs{\iint_{D\times E}\mean{h}^2\paren{\varphi(x,y)-\sum_{i=1}^{N}\varphi_1^i(x)\varphi_2^i(y)}\rmd x\rmd y} \\
         & \lesssim\theta_1+0+\theta_1\lesssim\theta_1,
      \end{align*}
  where we have used 
  \begin{equation*}
   \lim_{\varezero}\iint_{D\times E}\hep(x)\hep(y)\varphi_1^i(x)\varphi_2^i(y)\rmd x\rmd y
   =\int_{D}\mean{h}\varphi_1^i(x)\rmd x\int_{E}\mean{h}\varphi_2^i(y)\rmd y
  \end{equation*}    
   in the second inequality.
   From the arbitrariness of $\varphi$ and $\theta_1$, we  deduce that $\hep(x)\hep(y)\rightarrow\mean{h}^2$ weakly in $L^2(D\times E)$.
   For any bounded measurable subset $G\subset\rd\times\rd$ and $\theta_2>0$, we  choose finitely many pairwise almost disjoint closed rectangles $D_j\times E_j\subset G$, $j=1,2,\ldots,N(G,\theta_2)$ 
   such that $\abs{G\setminus\cup_{j=1}^{N}\paren{D_j\times E_j}}\leq\theta_2$. Since we have shown that $\hep(x)\hep(y)\rightarrow\mean{h}^2$ weakly in $L^2(D_j\times E_j)$ for each $j$, the assertion that $\hep(x)\hep(y)\rightarrow\mean{h}^2$ weakly in $\ag$ can be deduced by a standard approximation argument.
   
   In order to prove \eqref{eq:bounde domain periodic convergence}, it suffices to show that the following relations hold:
   \begin{gather}
      \lim_{\varezero}\iint_{G}\hep(x)\hep(y)\paren{\gep(x)}^2k(x,y)\rmd x\rmd y=
     \mean{h}^2\iint_{G}\paren{g(x)}^2k(x,y)\rmd x\rmd y,
     \label{eq:bounde domain periodic convergence 1}\\
     \lim_{\varezero}\iint_{G}\hep(x)\hep(y)\paren{\gep(y)}^2k(x,y)\rmd x\rmd y=
     \mean{h}^2\iint_{G}\paren{g(y)}^2k(x,y)\rmd x\rmd y,
     \label{eq:bounde domain periodic convergence 2} \\
     \lim_{\varezero}\iint_{G}\hep(x)\hep(y)\gep(x)\gep(y)k(x,y)\rmd x\rmd y=
     \mean{h}^2\iint_{G}g(x)g(y)k(x,y)\rmd x\rmd y.
     \label{eq:bounde domain periodic convergence 3}
   \end{gather}
   We first prove \eqref{eq:bounde domain periodic convergence 1}; \eqref{eq:bounde domain periodic convergence 2} follows by a similar argument.
   By the triangle inequality, we only need to prove that
   \begin{equation}\label{eq:first term from 1}
     \lim_{\varezero}\iint_{G}\hep(x)\hep(y)\paren{\paren{\gep(x)}^2-\paren{g(x)}^2}k(x,y)\rmd x\rmd y=0
   \end{equation}
   and that
    \begin{equation}\label{eq:second term from 1}
     \lim_{\varezero}\iint_{G}\paren{\hep(x)\hep(y)-\mean{h}^2}\paren{g(x)}^2k(x,y)\rmd x\rmd y=0.
   \end{equation}
   Since $G\subset\rd\times\rd$ is a bounded measurable subset, we may suppose that $G\subset D\times E$ for some closed rectangles $D$, $E\subset\rd$. For \eqref{eq:first term from 1}, from the uniform boundedness of $\hep$, $k(x,y)$ in $\linftyg$ and $\gep$ in $\ard$ , H\"older's inequality, and the fact that $\gep\rightarrow g$ strongly in $\ardloc$ as $\varezero$, it follows that
   \begin{align*}
     & \abs{\iint_{G}\hep(x)\hep(y)\paren{\paren{\gep(x)}^2-\paren{g(x)}^2}k(x,y)\rmd x\rmd y} \\
     &\lesssim\paren{\int_{D}\abs{\gep-g}^2}^{1/2}\paren{\int_{D}\abs{\gep+g}^2}^{1/2}
     \longrightarrow0,\quad \text{as}\hsp \varezero.
   \end{align*}
     For \eqref{eq:second term from 1}, note that $\paren{g(x)}^2k(x,y)\notin\ag$, thus it is not valid to deduce \eqref{eq:second term from 1} directly from the fact that $\hep(x)\hep(y)\rightarrow\mean{h}^2$ weakly in $\ag$. To overcome the difficulty, we introduce the truncated function $g_n(x):=\max\baren{-n,\min\baren{g(x),n}}$, which clearly satisfies that $g_n\in\linftyrd$ for each $n\in\mathbb{N^*}$ and that $g_n\rightarrow g$ strongly in $\ag$ as $n\to\infty$.
  Hence, we conclude that
  \begin{equation*}
    \lim_{\varezero}\abs{\iint_{G}\paren{\hep(x)\hep(y)-\mean{h}^2}\paren{g_n(x)}^2k(x,y)\rmd x\rmd y}=0,
  \end{equation*}
  and that for any $0<\theta_3<1$, there exists $N(\theta_3)>0$ such that for all $N<n\in\mathbb{N^*}$, 
  $\norm{g_n-g}_{\ard}<\theta_3$.
  Consequently, proceeding as in the proof of \eqref{eq:first term from 1} and employing the triangle inequality, for all $N<n\in\mathbb{N^*}$, we have
  \begin{align*}
    &\limsup_{\varezero}\abs{\iint_{G}\paren{\hep(x)\hep(y)-\mean{h}^2}\paren{g(x)}^2k(x,y)\rmd x\rmd y} \\
    &\leq  \limsup_{\varezero}\abs{\iint_{G}\paren{\hep(x)\hep(y)-\mean{h}^2}\paren{\paren{g(x)}^2-
    \paren{g_n(x)}^2} k(x,y)\rmd x\rmd y} \\
    &\quad + \lim_{\varezero}\abs{\iint_{G}\paren{\hep(x)\hep(y)-\mean{h}^2}\paren{g_n(x)}^2k(x,y)\rmd x\rmd y}\\
    &\leq C\paren{\int_{D}\abs{g-g_n}^2}^{1/2}\paren{\int_{D}\abs{g+g_n}^2}^{1/2}\leq C\theta_3,
  \end{align*}
  where the constant $C>0$ is independent of $n$ and $\theta_3$. Therefore, letting $\theta_3\rightarrow0$ yields \eqref{eq:second term from 1}. 
  
  We proceed to prove \eqref{eq:bounde domain periodic convergence 3}. Similar to the proof of \eqref{eq:first term from 1} and by applying the triangle inequality, we obtain
  \begin{align*}
    & \limsup_{\varezero}\abs{\iint_{G}\hep(x)\hep(y)\gep(x)\gep(y)k(x,y)\rmd x\rmd y-
     \mean{h}^2\iint_{G}g(x)g(y)k(x,y)\rmd x\rmd y} \\
     & \leq \limsup_{\varezero}\abs{\iint_{G}\hep(x)\hep(y)\paren{\gep(x)\gep(y)-g(x)g(y)}k(x,y)\rmd x\rmd y}\\
     & \quad+\lim_{\varezero}\abs{\iint_{G}\paren{\hep(x)\hep(y)-\mean{h}^2}g(x)g(y)k(x,y)\rmd x\rmd y} \\
     & \lesssim \limsup_{\varezero}\iint_{G}\abs{\gep(x)\gep(y)-g(x)g(y)}\rmd x\rmd y\\
     &\leq\limsup_{\varezero}\iint_{D\times E}\abs{\gep(x)\gep(y)-\gep(x)g(y)}\rmd x\rmd y
     +\limsup_{\varezero}\iint_{D\times E}\abs{\gep(x)g(y)-g(x)g(y)}\rmd x\rmd y\\
     &\lesssim\limsup_{\varezero}\norm{\gep}_{L^2(D)}\norm{\gep-g}_{L^2(E)}+
     \limsup_{\varezero}\norm{g}_{L^2(E)}\norm{\gep-g}_{L^2(D)}=0,
  \end{align*}
  where we have used the fact that $g(x)g(y)k(x,y)\in\ag$ in the second inequality. Clearly, this implies \eqref{eq:bounde domain periodic convergence 3}, thus together with \eqref{eq:bounde domain periodic convergence 1} and \eqref{eq:bounde domain periodic convergence 2} we complete the proof of \eqref{eq:bounde domain periodic convergence}. 
   \end{proof}
   \begin{remark}\label{remark:bounded domain random convergence}
     Under the same replacement of assumptions as in Remark \ref{remark:random weak convergence}, but now applied to Lemma \ref{lemma:bounde domain periodic convergence}, we can prove that as $\varezero$, almost surely,
     \begin{equation*}
       \hep(x,\omega)\hep(y,\omega)\to\e{\boldsymbol h}^2\quad \text{weakly in}\hsp \ag.  
     \end{equation*}
      Therefore, we can conclude that the limit,
     \begin{equation*}
       \lim_{\varezero}\iint_{G}\hep(x,\omega)\hep(y,\omega)\paren{\gep(x)-\gep(y)}^2k(x,y)\rmd x\rmd y=
     \e{\boldsymbol h}^2\iint_{G}\paren{g(x)-g(y)}^2k(x,y)\rmd x\rmd y,
     \end{equation*}
     holds almost surely.
   \end{remark} 
    
 \section{Homogenization for product-type oscillating coefficients}
 In this section, we present the  proof of Theorem \ref{thm:periodic product structure} using the Gamma-convergence framework.  As for Theorem \ref{thm:random product structure}, by applying Remark \ref{remark:random weak convergence} in place of Lemma \ref{lemma:periodic weak convergence},and Remark \ref{remark:bounded domain random convergence} in place of Lemma \ref{lemma:bounde domain periodic convergence}, we systematically substitute $\mean{\mu/\lambda}$ by $\e{\boldsymbol\mu/\boldsymbol\lambda}$ and $\mean{\mu}$ by $\e{\boldsymbol\mu}$, simultaneously interpreting all limits in the almost sure sense. We therefore omit the details.
 \begin{proof}[Proof of Theorem \ref{thm:periodic product structure}]
    For $u\in\ard$, we define the functional:
     \[
     \fep(u)=
   \begin{cases}
    (-\lvar u,u)_{\ardnu}+m(u,u)_{\ardnu}-2(f,u)_{\ardnu} & \quad\text{if}\hsp u\in\hnorm, \\
     +\infty & \quad \text{if}\hsp u\in\ard\setminus\hnorm.
   \end{cases}
   \]
   The proof of Theorem \ref{thm:periodic product structure} is organized into three steps.
   
    \textbf{Step 1.} In this step, we will prove that for each $\varepsilon>0$, there is an equivalence between  weak solutions of  equation \eqref{eq:uep satisfies equation} and  minimizers of the functional $\fep$. 
    
    For this purpose, we first prove that the functional $\fep$ has a unique minimizer and this minimizer is indeed the solution $\uep$ of \eqref{eq:uep satisfies equation}. We now claim that $\fep$ is continuous and strictly convex on $\hnorm$. Indeed, for $u$, $v\in\hnorm$, by a similar argument to Lemma \ref{lemma:equivalent norm}, we have
   \begin{equation*}
     \paren{-\lvar u, v}_{\ardnu}=\frac{1}{2}\iintrdd \paren{u(x)-u(y)}\paren{v(x)-v(y)} K(x,y)\blv(x,y)\rmd x\rmd y,
   \end{equation*} 
   which clearly implies $\paren{-\lvar u, v}_{\ardnu}=\paren{u,-\lvar v}_{\ardnu}$. Therefore, a straightforward calculation shows that for $u$, $v\in\hnorm$ and $t\in\r$,
   \begin{align*}
    \fep(u+tv)-\fep(u)&= 2t\paren{(-\lvar u,v)_{\ardnu}+m(u,v)_{\ardnu}+(f,v)_{\ardnu}}\\
    &\quad+t^2\paren{(-\lvar v,v)_{\ardnu}+m(v,v)_{\ardnu}}.
   \end{align*} 
   It implies that $\fep$ is continuous on $\hnorm$, due to the fact that
   \begin{equation*}
      \abs{\paren{-\lvar u, v}_{\ardnu}}\lesssim[u]_{\hnorm}[v]_{\hnorm}<\infty,
     \end{equation*} 
  which is an immediate consequence of Remark \ref{remark:equi norm product form} and H\"older's inequality.
   Moreover, for $u$, $v\in\hnorm$ with $u\neq v$ and  $t\in(0,1)$, in view of Lemma \ref{lemma:equivalent norm}, we have
   \begin{equation*}
   \begin{aligned}
      &\fep(tu+(1-t)v)-t\fep(u)-(1-t)\fep(v)\\
      & =t(t-1)\paren{(-\lvar(u-v),u-v)_{\ardnu}+m(u-v,u-v)_{\ardnu}}\\
      & <0,
      \end{aligned}
   \end{equation*}
   which implies that $\fep$ is strictly convex on $\hnorm$. 
   We proceed to show that $\fep$ is coercive in the sense that if $\norm{u}_{\hnorm}\rightarrow+\infty$, then $\fep(u)\rightarrow+\infty$. 
    According to Remark \ref{remark:equi norm product form}, the uniform positivity of $\nuep$, and Young's inequality, we  deduce from the definition of $\fep$ that
    \begin{equation}\label{eq:fep is coercive}
    \begin{aligned}
 \fep(u)&\geq C_3[u]_{\hnorm}^2+C_4\norm{u}_{\ard}^2-C_5\norm{f}_{\ard}^2\\
      &\geq \min\{C_3/2, C_4/2\} \norm{u}_{\hnorm}^2-C_5\norm{f}_{\ard}^2,
      \end{aligned}
    \end{equation}
    where the constants $C_3$, $C_4$, $C_5$ are positive and independent of $u$ and $\varepsilon$.
    This implies the coercivity of $\fep$. 
    Therefore, by a standard argument (see \cite[Corollary 3.23]{MR2759829} or \cite[Chapter 2, Proposition 1.2]{MR1727362}) on the existence and uniqueness of minimizers for functionals on reflexive Banach spaces,  we conclude from 
     the continuity, the strict convexity, and the coercivity of $\fep$ that for each $\varepsilon>0$, the functional $\fep$ admits its minimum at a unique point $\tilde{u}^\varepsilon\in\hnorm$.  
     
     Our next goal is to prove that $\tilde{u}^\varepsilon$ satisfies  equation \eqref{eq:uep satisfies equation}. This derivation is standard. Since  $\tilde{u}^\varepsilon$ is the unique minimizer of $\fep$, it satisfies that
   \begin{equation*}
     \frac{\rmd}{\rmd t}\fep(\tilde{u}^\varepsilon+tv)\Big|_{t=0}=0\quad \text{for all}\quad v\in\hnorm.
   \end{equation*}
   By the definition of $\fep$, this is equivalent to the following:
   \begin{equation*}
     m\paren{\tilde{u}^\varepsilon,v}_{\ardnu}-\paren{\lvar\tilde{u}^\varepsilon,v}_{\ardnu}=
     \paren{f,v}_{\ardnu}\quad \text{for all}\quad v\in\hnorm,
   \end{equation*}
   which implies that $\tilde{u}^\varepsilon$ is a solution of \eqref{eq:uep satisfies equation}.
   
  On the other hand, let $\uep\in \hnorm$ be a weak solution of  equation \eqref{eq:uep satisfies equation}.  For any fixed $v\in \hnorm$, testing \eqref{eq:uep satisfies equation} against $\uep-v$ and applying Young's inequality, it is straightforward to  derive $\fep(\uep)\leq\fep(v)$, which implies that $\uep$ is a minimizer of $\fep$.  
 
 \textbf{Step 2.} Next, we prove that, with respect to the $\ardloc\cap\ardw$ topology, $\fep$ $\Gamma$-converges to the functional $F^0(u)$ defined on $\ard$:
  \[
     F^0(u)=
   \begin{dcases}
    \frac{1}{2}\mean{\mu}^2(u,u)_K+\mean{\mu/\lambda}\intrd\paren{m\paren{u(x)}^2-2f(x)u(x)}\rmd x   &\quad\text{if}\hsp u\in\hnorm, \\
     +\infty &\quad\text{if}\hsp u\in\ard\setminus\hnorm.
   \end{dcases}
   \]
   Here, $\ardw$ is denoted by the space of square integrable functions equipped with the weak convergence topology, and  $(u,u)_K$ is defined in \eqref{eq:equai norm}.
   
   According to the definition of $\Gamma$-convergence, we need to prove the following: 
   \begin{itemize}
     \item $\Gamma$-$\liminf$ inequality: For all $v\in\ard$, and every sequence $\{\vep\}_{\varepsilon>0}\subset\ard$  converging to $v$ in $\ardloc\cap\ardw$ topology, it holds that
         \begin{equation}\label{eq:gamma liminf}
           F^0(v)\leq\liminf_{\varezero}\fep(\vep);
         \end{equation}
     \item $\Gamma$-$\limsup$ inequality: For all $v\in\ard$, there exists a sequence  $\{\vep\}_{\varepsilon>0}\subset\ard$  converging to $v$ in $\ardloc\cap\ardw$ topology, such that
         \begin{equation}\label{eq:gamma limsup}
           F^0(v)\geq\limsup_{\varezero}\fep(\vep).
         \end{equation}
   \end{itemize}
   We first prove the $\Gamma$-$\liminf$ inequality \eqref{eq:gamma liminf}. If $v\in\ard\setminus\hnorm$, for the sequence $\{\vep\}_{\varepsilon>0}$, we claim that 
   \begin{equation}\label{eq:vep hnorm liminf is infty}
     \liminf_{\varezero}\, \norm{\vep}_{\hnorm}=+\infty.
   \end{equation}
   In fact, if \eqref{eq:vep hnorm liminf is infty} is false, then there exists a subsequence
   $\{v^{\varepsilon_i}\}_{i=1}^\infty\subset\{v^\varepsilon\}_{\varepsilon>0}$ and a uniform constant $C>0$ such that $\norm{v^{\varepsilon_i}}_{\hnorm}\leq C$. Therefore, up to extracting a subsequence of $\{v^{\varepsilon_i}\}_{i=1}^\infty$ (without relabeling), we have $v^{\varepsilon_i}\rightarrow \tilde{v}$ weakly in $\hnorm$ (thus also in $\ard$) for some $\tilde{v}\in\hnorm$ as $i\rightarrow\infty$.
   Recall that $\vep$ converges weakly to $v$ in $\ard$. From the uniqueness of the weak limit, we deduce that $\tilde{v}=v$. Due to the weak lower semicontinuity of the $\hnorm$ norm, it follows that
   \begin{equation*}
     \norm{v}_{\hnorm}\leq\liminf_{i\rightarrow\infty}\,\norm{v^{\varepsilon_i}}_{\hnorm}\leq C,
   \end{equation*} 
   which contradicts our assumption that $v\in\ard\setminus\hnorm$. In view of 
   \eqref{eq:fep is coercive} and \eqref{eq:vep hnorm liminf is infty}, we obtain
   \begin{equation*}
     \liminf_{\varezero}\fep(\vep)\gtrsim\liminf_{\varezero}\,\norm{\vep}_{\hnorm}-
     C\norm{f}_{\ard}=+\infty=F^0(v).
   \end{equation*}
   We now turn to the case when $v\in\hnorm$. By the definitions of $\fep$ and $F^0$, it suffices to prove the following relations:
   \begin{equation}\label{eq:convergence of term f v}
     \lim_{\varezero}\,(f,\vep)_{\ardnu}=\mean{\mu/\lambda}(f,v),  
     \end{equation}
   \begin{equation}\label{eq:convergence of term v v}
     \liminf_{\varezero}\,(\vep,\vep)_{\ardnu}\geq\mean{\mu/\lambda}(v,v),  
     \end{equation} 
     \begin{equation}\label{eq:convergence of term main}
     \begin{aligned}
      \liminf_{\varezero}&\,\iintrdd\muep(x)\muep(y)\paren{\vep(x)-\vep(y)}^2K(x,y)\rmd y\rmd x\\
     &\geq \,\mean{\mu}^2\iintrdd\paren{v(x)-v(y)}^2K(x,y)\rmd y\rmd x.
      \end{aligned}
     \end{equation} 
     Recall that we have assumed that $\vep\rightarrow v$ in $\ardloc$ and that $f\in\ard$. Hence, \eqref{eq:convergence of term f v} follows directly from Lemma \ref{lemma:periodic weak convergence}. For the same reason, we see that
     \begin{equation*}
   \lim_{\varezero}\,(v,v)_{\ardnu}=\mean{\mu/\lambda}(v,v)=\lim_{\varezero}\,(v,\vep)_{\ardnu}.
     \end{equation*}
     Therefore,
     \begin{align*}
       0&\leq  \liminf_{\varezero}\,(v-\vep,v-\vep)_{\ardnu} \\
      & = \liminf_{\varezero}\,\paren{(\vep,\vep)_{\ardnu}+(v,v)_{\ardnu}-2(v,\vep)_{\ardnu}}\\
      & = \liminf_{\varezero}\,(\vep,\vep)_{\ardnu}-\mean{\mu/\lambda}(v,v),
     \end{align*}
      which implies \eqref{eq:convergence of term v v}.
      
  In order to prove \eqref{eq:convergence of term main}, we divide $\rd\times\rd$ into three subsets: $G_1^\delta$, $G_2^\delta$, and $G_3^\delta$, which are defined by the following:
   \begin{gather}
     G_1^\delta = \baren{(x,y)\in\rd\times\rd : K(x,y)\leq\delta^{-(d+\alpha)}, \abs{x}+\abs{y}<1/\delta}\label{eq:G1 delta}, \\
     G_2^\delta = \baren{(x,y)\in\rd\times\rd : K(x,y)>\delta^{-(d+\alpha)}, \abs{x}+\abs{y}<1/\delta}\label{eq:G2 delta},\\
     G_3^\delta = \baren{(x,y)\in\rd\times\rd : \abs{x}+\abs{y}\geq1/\delta}\label{eq:G3 delta}.
   \end{gather}
  It follows from the assumption \eqref{B2} that $K(x,y)<\infty$ for a.e. $(x,y)\in\rd\times\rd$. We thus derive 
  \begin{equation}\label{eq:G2G3 tend 0}
    \mathbbm{1}_{G_2^\delta\cup G_3^\delta}(x,y)\to 0\hsp\text{as}\; \delta\to0
  \end{equation}
   for a.e. $(x,y)\in\rd\times\rd$.
  For the kernel $K$ in Case 1, in view of \eqref{A3}, \eqref{eq:G2G3 tend 0}  follows by an entirely similar argument, and we therefore omit the details. 
  Hence, we conclude from  the dominated convergence theorem, $v\in\hnorm$, and \eqref{eq:equai norm} that for any $\kappa>0$, there exists $\delta>0$ such that
  \begin{equation}\label{eq:G2,G3small integral}
    \mean{\mu}^2\iint_{G_2^\delta\cup G_3^\delta}\paren{v(x)-v(y)}^2K(x,y)\rmd y\rmd x\leq\kappa.
  \end{equation}
  Consequently,
  \begin{equation}\label{eq:G2,G3 integral compare}
  \begin{aligned}
    \liminf_{\varezero}&\iint_{G_2^\delta\cup G_3^\delta}\muep(x)\muep(y)\paren{\vep(x)-\vep(y)}^2K(x,y)\rmd y\rmd x \\
   &\geq0 \geq\mean{\mu}^2\iint_{G_2^\delta\cup G_3^\delta}\paren{v(x)-v(y)}^2K(x,y)\rmd y\rmd x-\kappa.
    \end{aligned}
  \end{equation}
  Note that for $(x,y)\in G_1^\delta$, we have $K(x,y)\in L^\infty$. Therefore,
   we  deduce directly from Lemma \ref{lemma:bounde domain periodic convergence} that
  \begin{equation}\label{eq:G1 integral compare}
    \lim_{\varezero}\iint_{G_1^\delta}\muep(x)\muep(y)\paren{\vep(x)-\vep(y)}^2K(x,y)\rmd x\rmd y=\mean{\mu}^2\iint_{G_1^\delta}\paren{v(x)-v(y)}^2K(x,y)\rmd x\rmd y.
  \end{equation}
  By the arbitrariness of $\kappa$, combining \eqref{eq:G1 integral compare} with \eqref{eq:G2,G3 integral compare} gives the desired formula \eqref{eq:convergence of term main}. Hence, we complete the proof of the 
  $\Gamma$-$\liminf$ inequality \eqref{eq:gamma liminf}. 
  
  We  now proceed analogously to the proof of the $\Gamma$-$\limsup$ inequality \eqref{eq:gamma limsup}. For any $v\in\ard$, we choose the sequence functions $\vep\equiv v$. 
  If $v\in\ard\setminus\hnorm$, obviously $F^0(v)=\limsup_{\varezero}\fep(\vep)=+\infty$.
  If $v\in\hnorm$, we can respectively replace \eqref{eq:convergence of term f v}, \eqref{eq:convergence of term v v}, \eqref{eq:G2,G3 integral compare}, and \eqref{eq:G1 integral compare} with 
  \begin{equation*}
     \lim_{\varezero}\,(f,v)_{\ardnu}=\mean{\mu/\lambda}(f,v),  
     \end{equation*}
   \begin{equation*}
     \lim_{\varezero}\,(v,v)_{\ardnu}=\mean{\mu/\lambda}(v,v),  
     \end{equation*} 
    \begin{align*}
    \limsup_{\varezero}&\iint_{G_2^\delta\cup G_3^\delta}\muep(x)\muep(y)\paren{v(x)-v(y)}^2K(x,y)\rmd y\rmd x \\
   &\leq\kappa \leq\mean{\mu}^2\iint_{G_2^\delta\cup G_3^\delta}\paren{v(x)-v(y)}^2K(x,y)\rmd y\rmd x+\kappa,
    \end{align*}
 and
  \begin{equation*}
    \lim_{\varezero}\iint_{G_1^\delta}\muep(x)\muep(y)\paren{v(x)-v(y)}^2K(x,y)\rmd x\rmd y=\mean{\mu}^2\iint_{G_1^\delta}\paren{v(x)-v(y)}^2K(x,y)\rmd x\rmd y.
  \end{equation*}
  By the definitions of $\fep$ and $F^0$, putting together these four relations yields the $\Gamma$-$\limsup$ inequality \eqref{eq:gamma limsup}.
  
  \textbf{Step 3.} Similar to the argument in \textbf{Step 1} establishing  equivalence between  weak solutions of  equation \eqref{eq:uep satisfies equation} and  minimizers of the functional $\fep$, we can show that the functional $F^0$ also admits a unique minimizer $u^0$, which coincides with the solution of the limit equation 
   \eqref{eq:u0 satisfies equation}. From \cite[Theorem 1.21]{MR1968440}, we deduce that any limit point of $\{\uep\}_{\varepsilon>0}$ is $u^0$, and thus we have $\uep\rightarrow u^0$ weakly in $\ard$ and strongly in $\ardloc$ as $\varezero$. By Corollary \ref{corollary:a prior estimate}, we have $\norm{\uep}_{\hnorm}\leq C$ with a constant $C$ independent of $\varepsilon$. Therefore, up to extracting a subsequence  $\{\uepi\}_{i=1}^\infty$, we have $\uepi\to \tilde{u}$ weakly in $\hnorm$   (thus also in $\ard$) for some $\tilde{u}\in\hnorm$ as $i\to\infty$. From the uniqueness of the weak limit, we deduce that $\tilde{u}=u^0$. Hence, we conclude that $\uep\rightarrow u^0$ weakly in $\hnorm$ as $\varezero$. 
   
   We are left with the task of proving that $\uep\rightarrow u^0$ strongly in $\ard$ as $\varezero$.
We argue by contradiction. If false, there exists some constant $\tilde{C}>0$ and a subsequence
$\{\uepi\}_{i=1}^\infty$ such that $\norm{u^{\varepsilon_i}-u}_{\ard}\geq\tilde{C}$.
From the fact that $\uep\rightarrow u^0$ strongly in $\ardloc$ as $\varezero$, it follows that
for any $n\in\mathbb{N}^*$, there exists $\varepsilon(n)>0$ such that $\norm{\uep-u}_{L^2(B_n)}\leq\tilde{C}/2$ for any $\varepsilon<\varepsilon(n)$. For $\tilde{C}$ above, from the fact that $u\in\ard$, we can deduce  that there exists $N(\tilde{C})>0$ such that $\norm{u}_{L^2(\rd\setminus B_n)}\leq\tilde{C}/2$ for all $n>N$. Therefore, we  show that for all $\varepsilon_i<\varepsilon(N)$ and $n>N$,
\begin{align*}
 \norm{\uepi}_{L^2(\rd\setminus B_n)} &\geq\norm{\uepi-u}_{L^2(\rd\setminus B_n)}-\norm{u}_{L^2(\rd\setminus B_n)}\\
 &=\paren{\norm{\uepi-u}_{\ard}^2-\norm{\uepi-u}_{L^2(B_n)}^2}^{1/2}
 -\norm{u}_{L^2(\rd\setminus B_n)}\\
 &\geq\tilde{C}/3.
\end{align*} 
Observe that  $\norm{\uepi}_{L^2(\rd\setminus B_n)}$ is decreasing in $n$. Hence, we deduce that there exists some constant $C>0$ and a subsequence
$\{\uepi\}_{i=1}^\infty$ such that for any $n\in\mathbb{N}^*$ and $\varepsilon_i<\varepsilon(n)$, $\norm{\uepi}_{L^2(\rd\setminus B_n)}\geq C$.
This leads to 
\begin{align*}
    m(\uepi,\uepi)_{\ardnui}-m(\psi_n\uepi,\psi_n\uepi)_{\ardnui}
  &=m\int_{\rd\setminus B_{n/2}}\paren{1-\psi_n^2}\nuepi\paren{\uepi}^2 \\
  & \geq m\int_{\rd\setminus B_n}\nuepi\paren{\uepi}^2\\
  &\geq C,
\end{align*}
where the positive constant $C$ is independent of $n$ and  $\psi_n(x):=\psi(x/n)$ is the scaling of a standard smooth cutoff function $\psi$ supported in $B_1$ with $\psi\equiv1$ on $B_{1/2}$.
According to the definition of $\fep$, if we can establish the following two limits uniformly in $\varepsilon$:
\begin{gather}
  \lim_{n\rightarrow\infty}\abs{\paren{f,\psi_n\uep}_{\ardnu}-\paren{f,\uep}_{\ardnu}}=0,\label{eq:psi n limit 1}\\
 \lim_{n\rightarrow\infty}\abs{\paren{-\lvar\paren{\psi_n\uep},\psi_n\uep}_{\ardnu}
 -\paren{-\lvar\uep,\uep}_{\ardnu}}=0,
\label{eq:psi n limit 2}
  \end{gather}
  we will get $\fep(\psi_n\uep)<\fep(\uep)$ for sufficiently large $n$ and sufficiently small $\varepsilon$, which contradicts the fact that $\uep$ is the unique minimizer of $\fep$ for all $\varepsilon>0$. 
  \eqref{eq:psi n   limit 1} is an immediate consequence of the Lebesgue dominated convergence
theorem and H\"older's inequality due to the fact that $\uep$ is bounded uniformly in $\ard$ and that $\psi_n-1\rightarrow0$ pointwise as $n\rightarrow\infty$. 
We are left with proving the limit \eqref{eq:psi n limit 2}. In view of Lemma \ref{lemma:equivalent norm} and Remark \ref{remark:equi norm product form}, we obtain
\begin{equation}\label{eq:estimate two qep minus}
  \begin{aligned}
  &\abs{\paren{-\lvar\paren{\psi_n\uep},\psi_n\uep}_{\ardnu}
 -\paren{-\lvar\uep,\uep}_{\ardnu}}\\
  &=\abs{\paren{-\lvar\paren{\paren{\psi_n+1}\uep},\paren{\psi_n-1}\uep}_{\ardnu}}\\
  &\lesssim\maren{\paren{\psi_n+1}\uep}_{\hnorm}\maren{\paren{\psi_n-1}\uep}_{\hnorm}.
\end{aligned}
\end{equation}
By the definition of the seminorm in $\hnorm$, the triangle inequality, and  $\psi_n\in\ccinfty$, it follows that
\begin{equation}\label{eq:two qep minus term1}
  \maren{\paren{\psi_n+1}\uep}_{\hnorm}^2\lesssim\maren{\uep}_{\hnorm}^2+
  \iintrdd\frac{\abs{\psi_n(x)-\psi_n(y)}^2\abs{\uep(y)}^2}{\abs{x-y}^{d+\alpha}}\rmd x\rmd y.
\end{equation}
We estimate the last integral by splitting $\rd\times\rd$ into two regions: 
\begin{equation*}
  \{(x,y)\in\rd\times\rd : |x - y| \leq 1 \}\quad\text{and}\quad
 \{(x,y)\in\rd\times\rd : |x - y| \geq 1 \}.
 \end{equation*}
With respect to the estimate on the first region, using the fact that $\abs{\psi_n(x)-\psi_n(y)}\lesssim\abs{x-y}/n$, for $n>1$, we have
\begin{equation}\label{eq:two qep minus term1 term1}
  \begin{aligned}
  \intrd\int_{\abs{x-y}\leq1}\frac{\abs{\psi_n(x)-\psi_n(y)}^2\abs{\uep(y)}^2}{\abs{x-y}^{d+\alpha}}\rmd x\rmd y
  &\lesssim\frac{1}{n^2}\intrd\int_{\abs{x-y}\leq1}\frac{\abs{\uep(y)}^2}{\abs{x-y}^{d+\alpha-2}}
  \rmd x\rmd y\\
  &\lesssim\frac{1}{n^2}\intrd\abs{\uep(y)}^2\rmd y\int_{\abs{z}\leq1}\frac{1}{\abs{z}^{d+\alpha-2}}\rmd z\\
   &\lesssim\norm{\uep}_{\ard}^2.
   \end{aligned}
\end{equation}
 For $x$, $y\in\rd$ satisfying 
$\abs{x-y}\geq1$, , it follows from  $\abs{\psi_n}\leq1$ that
\begin{equation}\label{eq:two qep minus term1 term2}
  \begin{aligned}
  \intrd\int_{\abs{x-y}\geq1}\frac{\abs{\psi_n(x)-\psi_n(y)}^2\abs{\uep(y)}^2}{\abs{x-y}^{d+\alpha}}\rmd x\rmd y
  &\lesssim\intrd\int_{\abs{x-y}\geq1}\frac{\abs{\uep(y)}^2}{\abs{x-y}^{d+\alpha}}\rmd x\rmd y\\
  &\lesssim\intrd\abs{\uep(y)}^2\rmd y\int_{\abs{z}\geq1}\frac{1}{\abs{z}^{d+\alpha}}\rmd z\\
   &\lesssim\norm{\uep}_{\ard}^2.
   \end{aligned}
\end{equation}
  In view of \eqref{eq:two qep minus term1}--\eqref{eq:two qep minus term1 term2}, we obtain
  \begin{equation}\label{eq:two qep minus term1 estimate}
    \maren{\paren{\psi_n+1}\uep}_{\hnorm}\lesssim\norm{\uep}_{\hnorm}\leq C.
  \end{equation}
 Set $\rho_n:=\psi_n-1$, then $\rho_n$ tends to zero pointwise as $n\to\infty$ . In light of 
  \eqref{eq:estimate two qep minus} and \eqref{eq:two qep minus term1 estimate}, to establish the limit \eqref{eq:psi n limit 2}, it is sufficient to prove that the relation,
  \begin{equation}\label{eq:fn uep seminorm to 0}
    \lim_{n\rightarrow\infty}\maren{\rho_n\uep}_{\hnorm}=0,
  \end{equation}
  holds uniformly in $\varepsilon$. By the triangle inequality, it is easy to get
  \begin{equation*}
    \maren{\rho_n\uep}_{\hnorm}^2\lesssim
    \iintrdd\frac{\abs{\uep(x)-\uep(y)}^2\abs{\rho_n(x)}^2}{\abs{x-y}^{d+\alpha}}\rmd x\rmd y+\iintrdd\frac{\abs{\rho_n(x)-\rho_n(y)}^2\abs{\uep(y)}^2}{\abs{x-y}^{d+\alpha}}\rmd x\rmd y.
  \end{equation*}
  By the Lebesgue dominated convergence theorem and the fact that $\norm{\uep}_{\hnorm}\leq C$, it follows that
   \begin{equation}\label{eq:fn 1}
     \lim_{n\rightarrow\infty}\iintrdd\frac{\abs{\uep(x)-\uep(y)}^2\abs{\rho_n(x)}^2}{\abs{x-y}^{d+\alpha}}\rmd x\rmd y=0.
   \end{equation}
  Similar to \eqref{eq:two qep minus term1 term1} and \eqref{eq:two qep minus term1 term2}, we deduce that as $n\rightarrow\infty$,
  \begin{equation}\label{eq:fn 2 outside ball}
    \int_{\rd\setminus B_{n/2}}
    \abs{\uep(y)}^2\intrd\frac{\abs{\rho_n(x)-\rho_n(y)}^2}{\abs{x-y}^{d+\alpha}}\rmd x\rmd y
    \lesssim\int_{\rd\setminus B_{n/2}}
    \abs{\uep(y)}^2\rmd y\to0,
  \end{equation}
  and
  \begin{equation}\label{eq:fn 2 inside ball}
    \begin{aligned}
    &\int_{B_{n/2}}
    \abs{\uep(y)}^2\intrd\frac{\abs{\rho_n(x)-\rho_n(y)}^2}{\abs{x-y}^{d+\alpha}}\rmd x\rmd y\\
    &\lesssim\int_{B_{n/2}}\abs{\uep(y)}^2\paren{\int_{\rd\setminus B_n}\frac{1}{\abs{x-y}^{d+\alpha}}\rmd x
    +\frac{1}{n^2}\int_{B_n\setminus B_{n/2}}\frac{1}{\abs{x-y}^{d+\alpha-2}}\rmd x}\rmd y\\
    &\lesssim\int_{B_{n/2}}\abs{\uep(y)}^2\paren{\int_{\rd\setminus B_{n/2}}\frac{1}{\abs{z}^{d+\alpha}}\rmd z
    +\frac{1}{n^2}\int_{B_{3n/2}}\frac{1}{\abs{z}^{d+\alpha-2}}\rmd z}\rmd y\\
    &\lesssim \frac{1}{n^\alpha}\int_{B_{n/2}}\abs{\uep(y)}^2\rmd y
    \lesssim\frac{1}{n^\alpha}\norm{\uep}_{\ard}^2\to0.
    \end{aligned}
  \end{equation}
  Combining \eqref{eq:fn 1}--\eqref{eq:fn 2 inside ball}, we obtain the desired \eqref{eq:fn uep seminorm to 0}, which completes the proof of Theorem \ref{thm:periodic product structure}.
 \end{proof}
 \section{Homogenization for symmetric oscillating coefficients}
 In this section, we present the  proof of Theorem \ref{thm:periodic sym structure} by the compactness argument. For the proof of Theorem \ref{thm:random sym structure}, it suffices to replace \cite[Lemma 3.1]{MR2085184} with \cite[Lemma 3.1]{MR4000213} in the proof of Theorem \ref{thm:periodic sym structure} below, simultaneously interpreting all limits in the almost sure sense. Therefore, we omit the details.
 \begin{proof}[Proof of Theorem \ref{thm:periodic sym structure}]
 According to Corollary \ref{corollary:a prior estimate}, we obtain the uniform estimate
 	\begin{equation}\label{eq:10}
 		\|u^{\varepsilon}\|_{\hnorm}\lesssim \|f\|_{L^{2}\left(\mathbb{R}^{d}\right)}\lesssim C,
 	\end{equation}
 	where the constant $C$ does not depend on $\varepsilon$. By the weak compactness  and the compact embedding theorem for fractional Sobolev space  \cite[Theorem 6.13]{MR4567945}, there exists a subsequence, still denoted by $\{u^{\varepsilon}\}_{\varepsilon>0}$, and a function $\tilde{u}\in \hnorm$ such that 
 	\begin{equation*}
 		u^{\varepsilon} \rightarrow \tilde{u} \quad{\rm weakly \;in\;} \hnorm
 	\end{equation*}	
 and
 	\begin{equation*}
 		u^{\varepsilon} \rightarrow \tilde{u} \quad{\rm strongly \;in\;} \ardloc
 	\end{equation*}	
 	as $\varepsilon \rightarrow 0 $.
 	
 	Next, we aim to find the equation satisfied by the limit function $\tilde{u}$. To this aim,  testing the equation $-L^{\varepsilon}u^{\varepsilon}+mu^{\varepsilon}=f$ against $\varphi\in C_{c}^{\infty}\left(\mathbb{R}^{d}\right)$, we obtain
 	\begin{align}\label{130}
 		&\frac{1}{2}\iintrdd\left(u^{\varepsilon}(y)-u^{\varepsilon}(x)\right)
 \left(\varphi(y)-\varphi(x)\right)K(x,y)\Lambda^{\varepsilon}(x,y)\rmd y\rmd x+\int_{\mathbb R^d}\left(mu^{\varepsilon}\varphi-f\varphi\right)
 		\rmd x=0.
 	\end{align}
 	Since $u^{\varepsilon} \rightarrow \tilde{u}$ weakly in $L^{2}\left(\mathbb{R}^{d}\right)$ as $\varepsilon\rightarrow 0$, it follows that 
 	\begin{equation}\label{43}
 		\lim\limits_{\varepsilon\rightarrow 0}\int_{\mathbb R^d}u^{\varepsilon}\varphi\rmd x=\int_{\mathbb R^d}\tilde{u}\varphi\rmd x.
 	\end{equation} 
 	For the first term on the left hand side of \eqref{130}, we claim that as $\varepsilon\rightarrow 0$,
 \begin{equation}\label{44}
 	\begin{aligned}	
 	&\iintrdd\left(u^{\varepsilon}(y)-u^{\varepsilon}(x)\right)\left(\varphi(y)-\varphi(x)\right)
 K(x,y)\Lambda^{\varepsilon}(x,y)\rmd y\rmd x\\
 	\rightarrow&\iintrdd\left(\tilde{u}(y)-\tilde{u}(x)\right)\left(\varphi(y)-\varphi(x)\right)
 K(x,y)\overline{\Lambda}(x,y)\rmd y\rmd x.
 	\end{aligned}	
  \end{equation}		
 	Once we have proved  \eqref{44}, combining \eqref{130}--\eqref{44} yields that 
 	\begin{align*}
 		&\frac{1}{2}\iintrdd\left(\tilde{u}(y)-\tilde{u}(x)\right)\left(\varphi(y)-
 \varphi(x)\right)K(x,y)\overline{\Lambda}(x,y)\rmd y\rmd x\\&+\int_{\mathbb R^d}\left(mu^{0}\varphi-f\varphi\right)\rmd x=0,\quad \forall \varphi\in C_{c}^{\infty}\left(\mathbb{R}^{d}\right).
 	\end{align*}
 Since $C_{c}^{\infty}\left(\mathbb{R}^{d}\right)$ is dense in $\hnorm$, this implies that $\tilde{u}$ is a solution of $-L^{0}u^0+mu^{0}=f$.
 	By the uniqueness of solution of this equation, we conclude that $u^{\varepsilon}$ converges to $u^0$ weakly in $\hnorm$ and strongly in $\ardloc$ as $\varepsilon\rightarrow 0$.
 	
 	In order to show the \eqref{44}, we divide the integration area $\mathbb{R}^{d}\times \mathbb{R}^{d}$ into three parts in the same way as previously defined in \eqref{eq:G1 delta}--\eqref{eq:G3 delta}.
 We first deal with the integral over $G_{2}^{\delta}\cup G_{3}^{\delta}$. By virtue of \eqref{eq:s3}, H\"{o}lder's inequality, Lemma \ref{lemma:equivalent norm}, and \eqref{eq:10}, we have
  \begin{equation}\label{45}
 	\begin{aligned}	
 		&\left|\iint_{G_{2}^{\delta}\cup G_{3}^{\delta}}\left(u^{\varepsilon}(y)-u^{\varepsilon}(x)\right)
 \left(\varphi(y)-\varphi(x)\right)K(x,y)\Lambda^{\varepsilon}(x,y)\rmd y\rmd x\right|\\
 		&\lesssim [u^{\varepsilon}]_{\hnorm}[\varphi]_{H^{\alpha/2}(G_{2}^{\delta}\cup G_{3}^{\delta})}\lesssim [\varphi]_{H^{\alpha/2}(G_{2}^{\delta}\cup G_{3}^{\delta})}.
 	\end{aligned}	
 \end{equation}		
 	Similar to \eqref{eq:G2,G3small integral}, we have
 	\begin{equation}\label{46}
 		[\varphi]_{H^{\alpha/2}(G_{2}^{\delta}\cup G_{3}^{\delta})}\rightarrow0, \quad {\rm as\;}\delta\rightarrow0.
 	\end{equation}
 	Combining (\ref{45}) with (\ref{46}), we derive
 	\begin{equation}\label{47}
 		\lim\limits_{\delta\rightarrow0}\iint_{G_{2}^{\delta}\cup G_{3}^{\delta}}\left(u^{\varepsilon}(y)-u^{\varepsilon}(x)\right)
 \left(\varphi(y)-\varphi(x)\right)K(x,y)\overline{\Lambda}(x,y)\rmd y\rmd x=0.
 	\end{equation}
 	Similarly, we have
 	\begin{equation}\label{48}
 		\lim\limits_{\delta\rightarrow0}\iint_{G_{2}^{\delta}\cup G_{3}^{\delta}}\left(u^{0}(y)-u^{0}(x)\right)\left(\varphi(y)-\varphi(x)\right)K(x,y)
 \overline{\Lambda}(x,y)\rmd y\rmd x=0.
 	\end{equation}
 	Next, we deal with the integral over $G_{1}^{\delta}$. Our goal is to prove that as  $\varepsilon\rightarrow 0$, 
 	  \begin{equation}\label{49}
 		\begin{aligned}	
 		I^{\varepsilon}:=&\iint_{G_{1}^{\delta}}\left(u^{\varepsilon}(y)-u^{\varepsilon}(x)\right)
 \left(\varphi(y)-\varphi(x)\right)K(x,y)\Lambda^{\varepsilon}(x,y)\rmd y\rmd x\\
 		\rightarrow&\iint_{G_{1}^{\delta}}\left(u^{0}(y)-u^{0}(x)\right)
 \left(\varphi(y)-\varphi(x)\right)K(x,y)\overline{\Lambda}(x,y)\rmd y\rmd x=:I^{0}.
 		\end{aligned}	
 	\end{equation}	
 	In fact, by the triangle inequality, we arrive at
 	\begin{equation}\label{410}
 		\begin{aligned}	
 			\left|I^{\varepsilon}-I^{0}\right|
 			\leq&\iint_{G_{1}^{\delta}}\left|\left(u^{\varepsilon}(y)-u^{\varepsilon}(x)\right)-
 \left(u^{0}(y)-u^{0}(x)\right)\right|
 			\left|\varphi(y)-\varphi(x)\right|K(x,y)\Lambda^{\varepsilon}(x,y)\rmd y\rmd x\\
 			&+\left|\iint_{G_{1}^{\delta}}\left(\Lambda^{\varepsilon}(x,y)-\overline{\Lambda}(x,y)\right)
 \left(u^{0}(y)-u^{0}(x)\right)\left(\varphi(y)-\varphi(x)\right)K(x,y)\rmd y\rmd x\right|\\
 			=:&I_{1}+I_{2}.
 		\end{aligned}	
 	\end{equation}
 	We first analyze $I_{1}$. Note that 
 	\begin{equation}\label{eq:bound}
 		K(x,y)\leq C(\delta),\quad \forall\, (x,y)\in G_{1}^{\delta}.
 	\end{equation}
 	In view of \eqref{eq:10}, \eqref{eq:bound}, and  H\"{o}lder's inequality, we obtain
 	\begin{align*}
 		\abs{I_{1}}\lesssim &\iint_{G_{1}^{\delta}}\left(\left|u^{\varepsilon}(y)-u^{0}(y)\right|+
 \left|u^{\varepsilon}(x)-u^{0}(x)\right|\right)|\varphi(y)-\varphi(x)|\rmd y\rmd x\\
 		\lesssim &\left(\int_{\{|x|\leq \delta^{-1}\}}\left|u^{\varepsilon}(x)-u^{0}(x)\right|^{2}\rmd x\right)^{1/2},
 	\end{align*}
 	which along with the fact that $u^{\varepsilon} \rightarrow u^{0}$ strongly in $\ardloc$ yields that
 	\begin{equation}\label{412}
 		I_{1}\rightarrow0, \quad {\rm as\;}\varepsilon\rightarrow0.
 	\end{equation}
 	We proceed to estimate $I_{2}$. Since $\Lambda(x,y,\xi,\eta)$ is a Carath\'eodory function, if $G_{1}^{\delta}$ is Jordan measurable,  we can apply \cite[Lemma 3.1]{MR2085184} and the fact that  $\{\Lambda^{\varepsilon}\}$ is bounded uniformly in $\varepsilon$ in $L^{2}\left(G_{1}^{\delta}\right)$ to deduce that as $\varepsilon\rightarrow 0$,
 	\begin{equation}\label{413}
 		\Lambda^{\varepsilon}\rightarrow \overline{\Lambda} \quad {\rm weakly\; in}\;  L^{2}\left(G_{1}^{\delta}\right). 
 	\end{equation}
 If $G_1^\delta$ is not Jordan measurable, we can approximate $G_1^{\delta}$ by rectangles as in Lemma \ref{lemma:bounde domain periodic convergence} to obtain \eqref{413}.
 	By the fact that $u^{0}\in \hnorm$ and that $\varphi\in C_{c}^{\infty}\left(\mathbb{R}^{d}\right)$, it is easily seen that
 	\begin{equation*}
 		K(x,y)\left(u^{0}(y)-u^{0}(x)\right)\left(\varphi(y)-\varphi(x)\right)\in L^{2}\left(G_{1}^{\delta}\right).
 	\end{equation*}
 	Consequently, applying (\ref{413}) to $I_{2}$ we obtain
 	\begin{equation}\label{414}
 		I_{2}\rightarrow0, \quad {\rm as\;}\varepsilon\rightarrow0.
 	\end{equation}
 	It follows from (\ref{410}), (\ref{412}), and (\ref{414}) that (\ref{49}) holds. Hence, by combining (\ref{47})--(\ref{49}), we have establisded the previous claim \eqref{44}.
 	
 	It remains to show that $\|u^{\varepsilon}-u^{0}\|_{L^{2}\left(\mathbb{R}^{d}\right)}\rightarrow 0$ as $\varepsilon\rightarrow 0$. We argue by contradiction. If 
 	\begin{equation}\label{eq:strict}
 		\lim\limits_{\varepsilon\rightarrow 0}\|u^{\varepsilon}-u^{0}\|_{L^{2}\left(\mathbb{R}^{d}\right)}\neq 0,
 	\end{equation}	
 	there exists a subsequence, still denoted by $\{u^{\varepsilon}\}_{\varepsilon>0}$, such that 
 	\begin{equation}\label{416}
 		\lim\limits_{\varepsilon\rightarrow 0}\|u^{\varepsilon}\|_{L^{2}\left(\mathbb{R}^{d}\right)}>\|u^{0}\|_{L^{2}\left(\mathbb{R}^{d}\right)}.
 	\end{equation}
 Indeed, since $u^{\varepsilon}$ converges to $u^{0}$ weakly in $\ard$, applying  the weak lower semicontinuity of the $\ard$ norm, we get
 \begin{equation*}
 	\liminf\limits_{\varepsilon\rightarrow 0}\|u^{\varepsilon}\|_{L^{2}\left(\mathbb{R}^{d}\right)}\geq \|u^{0}\|_{L^{2}\left(\mathbb{R}^{d}\right)}.
 \end{equation*}	
 This along with \eqref{eq:strict} directly implies \eqref{416}. On the other hand, we have 
 	\begin{align}\label{eq:positive}
 		0\leq \left(-L^{\varepsilon}(u^{\varepsilon}-u^{0}),u^{\varepsilon}-u^{0}\right)=
 \left(-L^{\varepsilon}u^{\varepsilon},u^{\varepsilon}\right)+
 2\left(L^{\varepsilon}u^{\varepsilon},u^{0}\right)-\left(L^{\varepsilon}u^{0},u^{0}\right).
 	\end{align}
 	We now claim that
 	\begin{align}\label{418}
 		&\liminf\limits_{\varepsilon\rightarrow 0}\left(-L^{\varepsilon}u^{\varepsilon},u^{\varepsilon}\right)\geq \left(-L^{0}u^{0},u^{0}\right).
 	\end{align}
 	Indeed, it follows from $-L^{\varepsilon}u^{\varepsilon}+mu^{\varepsilon}=f=-L^{0}u^{0}+mu^{0}$ that
 	\begin{align}\label{mli}
 		\lim\limits_{\varepsilon\rightarrow0}\left(L^{\varepsilon}u^{\varepsilon},u^{0}\right)=	\lim\limits_{\varepsilon\rightarrow0}\left(m(u^{\varepsilon},u^{0})-m(u^{0},u^{0})+
 \left(L^{0}u^{0},u^{0}\right)\right)=\left(L^{0}u^{0},u^{0}\right),
 	\end{align}
 	where we used the fact that $u^{\varepsilon} \rightarrow u^{0}$ weakly in $L^{2}\left(\mathbb{R}^{d}\right)$ in the last equality. We proceed to deal with the last term on the right hand side of \eqref{eq:positive}. Since $C_{c}^{\infty}\left(\mathbb{R}^{d}\right)$ is dense in $\hnorm$,  for any $\eta>0$, there exists $u_{\eta}\in C_{c}^{\infty}\left(\mathbb{R}^{d}\right)$ such that
 	\begin{align}\label{420}
 		\|u^{0}-u_{\eta}\|_{\hnorm}<\eta.
 	\end{align}
 	Applying the triangle inequality and H\"{o}lder's inequality , in combination with Lemma \ref{lemma:equivalent norm}, \eqref{44}, and \eqref{420}, we derive
 	\begin{equation}\label{fli}
 		\begin{aligned}	
 \limsup_{\varepsilon\rightarrow0}\left|\left(L^{\varepsilon}u^{0},u^{0}\right)-
 \left(L^{0}u^{0},u^{0}\right)\right|
 &\leq \limsup\limits_{\varepsilon\rightarrow0}\left(\left|\left(L^{\varepsilon}u^{0},u^{0}-
 u_{\eta}\right)\right|+
 \left|\left(L^{\varepsilon}u^{0},u_{\eta}\right)-\left(L^{0}u^{0},u^{0}\right)\right|\right)\\
 &\lesssim [u^{0}]_{\hnorm}[u^{0}-u_{\eta}]_{\hnorm}+\left|\left(L^{0}u^{0},u_{\eta}\right)-
 \left(L^{0}u^{0},u^{0}\right)\right|\\
 &\lesssim[u^{0}]_{\hnorm}[u^{0}-u_{\eta}]_{\hnorm}\lesssim \eta.
 		\end{aligned}	
 	\end{equation}
 	By the arbitrariness of $\eta>0$, \eqref{418} is an immediate consequence of \eqref{eq:positive}, \eqref{mli}, and \eqref{fli}. Combining \eqref{416} with \eqref{418} for the same subsequence we obtain
 	\begin{equation}\label{con}
 	\lim_{\varepsilon\rightarrow 0}
 		m\left(u^{\varepsilon},u^{\varepsilon}\right)-
 \left(L^{\varepsilon}u^{\varepsilon},u^{\varepsilon}\right)
 		>m\left(u^{0},u^{0}\right)-\left(L^{0}u^{0},u^{0}\right)=\left(f,u^{0}\right),
 	\end{equation}
 	where we have used the fact that $u^{0}$ satisfies $-L^{0}u^{0}+mu^{0}=f$ in the last equality. On the other hand, testing $-L^{\varepsilon}u^{\varepsilon}+mu^{\varepsilon}=f$ against $u^{\varepsilon}$ and passing to the limit as $\varepsilon \rightarrow 0$, we have
 	\begin{equation}\label{423}
 		\lim_{\varepsilon\rightarrow 0}
 		m\left(u^{\varepsilon},u^{\varepsilon}\right)-
 \left(L^{\varepsilon}u^{\varepsilon},u^{\varepsilon}\right)
 		=\lim_{\varepsilon\rightarrow 0}\left(f,u^{\varepsilon}\right)=\left(f,u^{0}\right),
 	\end{equation}
 	where we have used the fact that $u^{\varepsilon}$ converges to $u^{0}$ weakly in $L^{2}\left(\mathbb{R}^{d}\right)$ and that $f\in L^{2}\left(\mathbb{R}^{d}\right)$ in the last equality.
 	This contradicts (\ref{con}). Thus, we conclude that
 	$u^{\varepsilon} \rightarrow u^{0}$ strongly in $L^{2}\left(\mathbb{R}^{d}\right)$ as $\varepsilon \rightarrow 0 $, which completes the proof of Theorem \ref{thm:periodic sym structure}.
 \end{proof}

 \section{Homogenization of nonlinear nonlocal problems}
 In this section, we focus on proving Theorems \ref{thm1} and \ref{thm:nonlinear random}. Following the approach used in the homogenization of the linear problems, we present only the proof for the case of periodic symmetric coefficient. The proofs for the remaining cases can be obtained by straightforward adaptations, and are therefore omitted. 
  \begin{proof}[Proof of Theorem \ref{thm1}]	
 	\textbf{Step 1.} We prove the existence and uniqueness of the weak solution of (\ref{52}) by employing a direct method in the calculus of variations. For this purpose, we consider the following the functional
 	on $W^{\alpha/p,p}\left(\mathbb{R}^{d}\right)$
 	\begin{align}\label{varfun}
 		\jep(u)=\frac{1}{2}\iintrdd K(x,y)\Phi\paren{u(y)-u(x)}\Lambda^{\varepsilon}(x,y)\rmd y\rmd x+\frac{m}{p}\int_{\mathbb R^d}\abs{u}^p\rmd x-\intrd fu\rmd x.
 	\end{align}
 	According to the assumption that $\Phi\in C^1(\r)$ is strict convex, it is easily seen that the functional $\jep$ is continuous and strict convex on $\wap$.  
 In view of \eqref{1}, \eqref{eq:Phi condition}, \eqref{eq:s3}, and Young's inequality, we obtain
 \begin{equation*}
   \jep(u)\geq C(m,p, \gamma, \Upsilon_1, \Upsilon_2)\norm{u}^p_{\wap}-C(m,p)\norm{f}^{p'}_{\qrd}.
 \end{equation*}
 It clearly implies that $J$ is coercive on $\wap$.
 Therefore, for each $\varepsilon>0$,  the functional $\jep$ admits a unique minimizer $\tilde{u}^\varepsilon\in\wap$, which satisfies
 \begin{equation}\label{eq:jep minimizer}
    \begin{aligned}
 		0&=\frac{\rmd}{\rmd t}\jep(\tilde{u}^\varepsilon+tv)\Big|_{t=0}\\
 		&=\frac{1}{2}\iintrdd K(x,y)\Phi'\paren{\tilde{u}^{\varepsilon}(y)-\tilde{u}^{\varepsilon}(x)}
 \left(v(y)-v(x)\right)\Lambda^{\varepsilon}(x,y)\rmd y\rmd x\\
 		&\quad +m\intrd \abs{\tilde{u}^{\varepsilon}}^{p-2}\tilde{u}^{\varepsilon}v\rmd x-\intrd fv\rmd x
 	\end{aligned}
 \end{equation}
 	for all $v\in\wap$.  Observe that for any $u$, $v\in W^{\alpha/p,p}\left(\mathbb{R}^{d}\right)$, by exchanging the role of the variables and using the symmetry of $K(x,y)$ and $\Lambda^{\varepsilon}(x,y)$, together with the fact that $\Phi'$ is an odd function, we obtain
 	\begin{equation}\label{4}
	\begin{aligned}	
		&\quad-\iintrdd K(x,y)\Phi'\left(u(y)-u(x)\right)v(x)\blv(x,y)\rmd y\rmd x\\
		&=-\iintrdd K(x,y)\Phi'\left(u(x)-u(y)\right)v(y)\blv(x,y)\rmd y\rmd x\\
		&=\frac{1}{2}\iintrdd K(x,y)\Phi'\left(u(y)-u(x)\right)\left(v(y)-v(x)\right)\blv(x,y)\rmd y\rmd x.
	\end{aligned}	
\end{equation}
Hence, combining \eqref{eq:jep minimizer} with \eqref{4}  shows that $\tilde{u}^\varepsilon$ is  a weak solution of equation (\ref{52}).

 On the other hand, let $\uep\in\wap$ and $\vep\in \wap$
   be two  solutions of \eqref{52}. Testing (\ref{52}) against $\uep-\vep$ for $\uep$ and for $\vep$, and subtracting the resulting identities, we have 
 	\begin{align*}
 	  &\frac{1}{2}\iintrdd K(x,y)\big(\Phi'\paren{\uep(y)-\uep(x)}-\Phi'\paren{\vep(y)-\vep(x)}\big)\\
 &\quad\times\big(\paren{\uep(y)-\uep(x)}-\paren{\vep(y)-\vep(x)}\big)
\blv(x,y)\rmd y\rmd x\\
 		&\quad +m\intrd \big(\abs{\uep(x)}^{p-2}\uep(x)-\abs{\vep(x)}^{p-2}\vep(x)\big)\paren{\uep(x)-\vep(x)}\rmd x=0.
 	\end{align*}
 Observe that $(\Phi'(s)-\Phi'(t))(s-t)>0$ for all $s$, $t\in\r$, $s\neq t$ and $K(x,y)>0$, $\blv(x,y)>0$ for \text{a.e.} $x$, $y\in\rd$. Hence, the above equality implies $\uep=\vep$ almost everywhere. 
 
 \textbf{Step 2.} Testing equation (\ref{52}) against $u^{\varepsilon}$, by \eqref{4}, we obtain	
 \begin{equation}\label{5}
	\begin{aligned}	
		&\frac{1}{2}\iintrdd K(x,y)\Phi'\paren{\uep(y)-\uep(x)}\paren{\uep(y)-\uep(x)}
\blv(x,y)\rmd y\rmd x\\
 		&\quad +m\intrd \abs{\uep(x)}^p\rmd x=\intrd f(x)\uep(x)\rmd x.
	\end{aligned}	
\end{equation}  
 In view of \eqref{1}--\eqref{eq:Phi condition} and \eqref{eq:s1}--\eqref{eq:s3}, together with
 Young's inequality, we derive
 \begin{equation}\label{10}
 		\|u^{\varepsilon}\|_{W^{\alpha/p,p}\left(\mathbb{R}^{d}\right)}\lesssim\|f\|_{L^{p'}
 \left(\mathbb{R}^{d}\right)}^{p'/p}\leq C,
 	\end{equation}
 	where the constant $C$ does not depend on $\varepsilon$. Therefore, there exist a subsequence,  still denoted by $\{u^{\varepsilon}\}_{\varepsilon>0}$, and a function $\tilde{u}\in W^{\alpha/p,p}\left(\mathbb{R}^{d}\right)$, such that 
 	\begin{equation*}
 		u^{\varepsilon} \rightarrow \tilde{u} \quad{\rm weakly \;in\;} W^{\alpha/p,p}\left(\mathbb{R}^{d}\right)
 	\end{equation*}	
 and
 	\begin{equation*}
 		u^{\varepsilon} \rightarrow \tilde{u} \quad{\rm strongly \;in\;} \prdloc
 	\end{equation*}	
 	as $\varepsilon \rightarrow 0 $.
 	 Our aim in this step is to find the equation satisfied by the limit function $\tilde{u}$. To this aim,  testing equation (\ref{52}) against $\varphi\in C_{c}^{\infty}\left(\mathbb{R}^{d}\right)$ and using (\ref{4}), we obtain
 		\begin{equation}\label{weaks}
 		\begin{aligned}	
 			&\frac{1}{2}\iintrdd K(x,y)\Phi'(\uep(y)-\uep(x))\left(\varphi(y)-\varphi(x)\right)
 \Lambda^{\varepsilon}(x,y)\rmd y\rmd x\\
 			&+\int_{\mathbb R^d}\left(m\abs{\uep}^{p-2}\uep\varphi-f\varphi\right)\rmd x=0.
 		\end{aligned}	
 	\end{equation} 
 	For the second term on the left hand side of \eqref{weaks}, we claim that
 	\begin{equation}\label{120}
 		\lim\limits_{\varepsilon\rightarrow 0}\int_{\mathbb R^d}\abs{\uep}^{p-2}\uep\varphi\rmd x= \int_{\mathbb R^d}\abs{\tilde{u}}^{p-2}\tilde{u}\varphi\rmd x.
 	\end{equation}
 	In fact, when $1<p\leq 2$, by virtue of \eqref{eq:Phi derivation condition} and H\"{o}lder's inequality, we derive
 	\begin{align*}
 		\left|\int_{\mathbb R^d}\abs{\uep}^{p-2}\uep\varphi\rmd x- \int_{\mathbb R^d}\abs{\tilde{u}}^{p-2}\tilde{u}\varphi\rmd x\right|
 		&\lesssim\int_{\mathbb R^d}\left|u^{\varepsilon}-\tilde{u}\right|^{p-1}|\varphi|\rmd x\\
 		&\lesssim \left(\int_{\rm{supp}\; \varphi}\left|u^{\varepsilon}-\tilde{u}\right|^{p}\rmd x\right)^{1/p'}\left(\int_{\mathbb R^d}|\varphi|^{p}\rmd x\right)^{1/p},
 	\end{align*}
 	which together with the fact that $u^{\varepsilon}\rightarrow \tilde{u}$ strongly in $\prdloc$ shows \eqref{120}. When $p>2$, taking advantage of (\ref{eq:Phi derivation condition}), \eqref{10}, and H\"{o}lder's inequality, we have
 	\begin{align*}
 		&\left|\int_{\mathbb R^d}\abs{\uep}^{p-2}\uep\varphi\rmd x- \int_{\mathbb R^d}\abs{\tilde{u}}^{p-2}\tilde{u}\varphi\rmd x\right|\\
 		&\lesssim\left[\left(\int_{\mathbb R^d}|u^{\varepsilon}|^{p}\rmd x\right)^{(p-2)/p}+\left(\int_{\mathbb R^d}|\tilde{u}|^{p}\rmd x\right)^{(p-2)/p}\right]\left(\int_{\rm{supp}\; \varphi}\left|u^{\varepsilon}-\tilde{u}\right|^{p}\rmd x\right)^{1/p}\left(\int_{\mathbb R^d}|\varphi|^{p}\rmd x\right)^{1/p}\\
 		&\lesssim \left(\int_{\rm{supp}\; \varphi}\left|u^{\varepsilon}-\tilde{u}\right|^{p}\rmd x\right)^{1/p}\left(\int_{\mathbb R^d}|\varphi|^{p}\rmd x\right)^{1/p},
 	\end{align*}
 	which along with the fact that $u^{\varepsilon} \rightarrow \tilde{u}$ strongly in $\prdloc$ yields \eqref{120}.
 	
 	For the first term on the left hand side of \eqref{weaks}, we claim that as $\varepsilon\rightarrow 0$, 
 \begin{equation}\label{110}
 	\begin{aligned}	
 		&\iintrdd K(x,y)\Phi'(\uep(y)-\uep(x))\left(\varphi(y)-\varphi(x)\right)
 \Lambda^{\varepsilon}(x,y)\rmd y\rmd x\\
 	\rightarrow&\iintrdd K(x,y)\Phi'(\tilde{u}(y)-\tilde{u}(x))\left(\varphi(y)-\varphi(x)\right)
 \overline{\Lambda}(x,y)\rmd y\rmd x.
 	\end{aligned}	
 \end{equation} 	
 	Once we have proved \eqref{110}, combining (\ref{weaks})--(\ref{110}) yields that 
 	\begin{align*}
 		&\frac{1}{2}\iintrdd K(x,y)\Phi'(\tilde{u}(y)-\tilde{u}(x))
 \left(\varphi(y)-\varphi(x)\right)\overline{\Lambda}(x,y)\rmd y\rmd x\\
 		&+\int_{\mathbb R^d}\left(m\abs{\tilde{u}}^{p-2}\tilde{u}\varphi-f\varphi\right)\rmd x=0,\quad \forall \;\varphi\in C_{c}^{\infty}\left(\mathbb{R}^{d}\right).
 	\end{align*}
 	Due to the density of $C_{c}^{\infty}\left(\mathbb{R}^{d}\right)$ in $W^{\alpha/p,p}\left(\mathbb{R}^{d}\right)$, we obtain that $\tilde{u}$ is a solution of $-L^0_{\Phi}u+m\abs{u}^{p-2}u=f$.
 	By the uniqueness of solution of this equation, we conclude that $u^{\varepsilon}$ converges to $u^0$ weakly in $W^{\alpha/p,p}\left(\mathbb{R}^{d}\right)$ and strongly in $\prdloc$ as $\varepsilon\rightarrow 0$.
 	
 	In order to show the \eqref{110}, we apply a similar argument to \eqref{44}. By virtue of \eqref{eq:s3}, (\ref{1}),  \eqref {eq:Phi derivation condition}, H\"{o}lder's inequality, and (\ref{10}), we have
 	\begin{align*}
 		\left|\int_{G_{2}^{\delta}\cup G_{3}^{\delta}}K(x,y)\Phi'(\uep(y)-\uep(x))
 \left(\varphi(y)-\varphi(x)\right)
 \Lambda^{\varepsilon}(x,y)\rmd y\rmd x\right|
 		\lesssim [\varphi]_{W^{\alpha/p,p}\left(G_{2}^{\delta}\cup G_{3}^{\delta}\right)},
 	\end{align*}
 which together with that fact that  $[\varphi]_{W^{\alpha/p,p}\left(G_{2}^{\delta}\cup G_{3}^{\delta}\right)}\rightarrow0$ as $\delta\rightarrow0$ yields 
 	\begin{equation}\label{13}
 		\lim\limits_{\delta\rightarrow0}\int_{G_{2}^{\delta}\cup G_{3}^{\delta}}K(x,y)\Phi'(\uep(y)-\uep(x))
 \left(\varphi(y)-\varphi(x)\right)\Lambda^{\varepsilon}(x,y)\rmd y\rmd x=0.
 	\end{equation}
 	Similarly, we have
 	\begin{equation}\label{14}
 		\lim\limits_{\delta\rightarrow0}\int_{G_{2}^{\delta}\cup G_{3}^{\delta}}K(x,y)\Phi'(\tilde{u}(y)-\tilde{u}(x))\left(\varphi(y)-\varphi(x)\right)
 \overline{\Lambda}(x,y)\rmd y\rmd x=0.
 	\end{equation}
 	We proceed to prove that as $\varepsilon\rightarrow0$,
 	 \begin{equation}\label{15}
 		\begin{aligned}	
 			J_{1}^{\varepsilon}:&=\int_{G_{1}^{\delta}}K(x,y)\Phi'(\uep(y)-\uep(x))
 \left(\varphi(y)-\varphi(x)\right)\Lambda^{\varepsilon}(x,y)\rmd y\rmd x\\
 			&\rightarrow\int_{G_{1}^{\delta}}K(x,y)\Phi'(\tilde{u}(y)-\tilde{u}(x))
 \left(\varphi(y)-\varphi(x)\right)\overline{\Lambda}(x,y)\rmd y\rmd x=:J_{1}^{0}.
 		\end{aligned}	
 	\end{equation} 
 	In fact, by the triangle inequality, we arrive at
 	 \begin{equation}\label{16}
 		\begin{aligned}	
 			\left|J_{1}^{\varepsilon}-J_{1}^{0}\right|		&\leq\int_{G_{1}^{\delta}}K(x,y)\left|\Phi'(\uep(y)-\uep(x))-
 \Phi'(\tilde{u}(y)-\tilde{u}(x))\right|
 			\left|\varphi(y)-\varphi(x)\right|\Lambda^{\varepsilon}(x,y)\rmd y\rmd x\\
 			&\quad+\left|\int_{G_{1}^{\delta}}K(x,y)\left(\Lambda^{\varepsilon}(x,y)-
 \overline{\Lambda}(x,y)\right)\Phi'
 \left(\tilde{u}(y)-\tilde{u}(x)\right)\left(\varphi(y)-\varphi(x)\right)\rmd y\rmd x\right|\\
 			&=:I_{1}+I_{2}.
 		\end{aligned}	
 	\end{equation} 
 	 Applying \cite[Lemma 3.1]{MR2085184} and the fact that  $\{\Lambda^{\varepsilon}\}$ is bounded uniformly in $\varepsilon$ in $L^{p}\left(G_{1}^{\delta}\right)$, we deduce that
 	\begin{equation}\label{17}
 		\Lambda^{\varepsilon}\rightarrow \overline{\Lambda} \quad {\rm weakly\; in}\;  L^{p}\left(G_{1}^{\delta}\right),
 	\end{equation}
 	as $\varepsilon\rightarrow 0$. Combining (\ref{17}) with the fact that
 		$K(x,y)\Phi'(\tilde{u}(y)-\tilde{u}(x))\left(\varphi(y)-\varphi(x)\right)\in L^{p'}\left(G_{1}^{\delta}\right)$, we obtain
 	\begin{equation}\label{18}
 		I_{2}\rightarrow 0 \quad {\rm as \;}\varepsilon\rightarrow 0.
 	\end{equation}
 	Next, we analyze $I_{1}$ under two cases. In the first case, when $1<p\leq 2$, by using \eqref{eq:s3}, \eqref{eq:Phi derivation condition}, the fact that $K(x,y)\leq C$ for all $(x,y)\in G_{1}^{\delta}$, and H\"{o}lder's inequality, we have
 	\begin{align*}
 		I_{1}
 		&\lesssim \int_{G_{1}^{\delta}}\left(\left|u^{\varepsilon}(y)-\tilde{u}(y)\right|^{p-1}+
 \left|u^{\varepsilon}(x)-\tilde{u}(x)\right|^{p-1}\right)|\varphi(y)-\varphi(x)|\rmd y\rmd x\\
 		&\lesssim \left(\int_{|x|\leq \delta^{-1}}\left|u^{\varepsilon}(x)-\tilde{u}(x)\right|^{p}\rmd x\right)^{1/p'},
 	\end{align*}
 	which along with the fact that $u^{\varepsilon} \rightarrow \tilde{u}$ strongly in $\prdloc$ yields 
 \begin{equation}\label{19}
 		I_{1}\rightarrow 0 \quad {\rm as \;}\varepsilon\rightarrow 0.
 	\end{equation}
 	In the second case, when $p>2$, employing \eqref{eq:s3}, (\ref{1}), and \eqref{eq:Phi derivation condition}, we have
 	\begin{align*}
 		I_{1}&\lesssim \int_{G_{1}^{\delta}}\frac{\left|u^{\varepsilon}(y)-u^{\varepsilon}(x)\right|^{p-2}
 \left|u^{\varepsilon}(y)-u^{\varepsilon}(x)-\left(\tilde{u}(y)-\tilde{u}(x)\right)
 \right||\varphi(y)-\varphi(x)|}{|x-y|^{d+\alpha}}\rmd y\rmd x\\
 		&\quad+\int_{G_{1}^{\delta}}\frac{\left|\tilde{u}(y)-\tilde{u}(x)\right|^{p-2}
 \left|u^{\varepsilon}(y)-u^{\varepsilon}(x)-\left(\tilde{u}(y)-\tilde{u}(x)\right)
 \right||\varphi(y)-\varphi(x)|}{|x-y|^{d+\alpha}}\rmd y\rmd x\\
 		&=:J_{1}+J_{2}.
 	\end{align*}
 	By the triangle inequality, H\"{o}lder's inequality, and \eqref{10}, we have
 	\begin{align*}
 		J_{1}&\lesssim \left(\int_{G_{1}^{\delta}}\frac{\left|u^{\varepsilon}(y)-
 u^{\varepsilon}(x)\right|^{p}}{|x-y|^{d+\alpha}}\rmd y\rmd x\right)^{(p-2)/p}
 		\left(\int_{G_{1}^{\delta}}\frac{|\varphi(y)-\varphi(x)|^{p}}{|x-y|^{d+\alpha}}\rmd y\rmd x\right)^{1/p}\\	&\quad\times\left(\int_{G_{1}^{\delta}}\frac{\left(\left|u^{\varepsilon}(y)-\tilde{u}(y)\right|+
 \left|u^{\varepsilon}(x)-\tilde{u}(x)\right|\right)^{p}}{|x-y|^{d+\alpha}}\rmd y\rmd x\right)^{1/p}\\
 &\lesssim\left(\int_{G_{1}^{\delta}}\left(\left|u^{\varepsilon}(y)-\tilde{u}(y)\right|+
 \left|u^{\varepsilon}(x)-\tilde{u}(x)\right|\right)^{p}\rmd y\rmd x\right)^{1/p},
 	\end{align*}
 	which together with the fact that $u^{\varepsilon} \rightarrow \tilde{u}$ strongly in $\prdloc$ yields that
 	\begin{equation*}
 		J_{1}\rightarrow 0 \quad {\rm as \;}\varepsilon\rightarrow 0.
 	\end{equation*}
 	Similarly, we have 
 	\begin{equation*}
 		J_{2}\rightarrow 0 \quad {\rm as \;}\varepsilon\rightarrow 0.
 	\end{equation*}
 	Combining the above relations gives (\ref{19}). It follows from (\ref{16}), (\ref{18}), and (\ref{19}) that (\ref{15}) holds. Therefore, our claim
 	\eqref{110} is an immediate consequence of
 	(\ref{13})--(\ref{15}).
 	
 	\textbf{Step 3.}  We are left with the task of proving that $\|u^{\varepsilon}-u^{0}\|_{L^{p}\left(\mathbb{R}^{d}\right)}\rightarrow 0$ as $\varepsilon\rightarrow 0$. 
 	We argue by contradiction. If false,
 	there exists a subsequence, still denoted by $\{u^{\varepsilon}\}_{\varepsilon>0}$, such that 
 	\begin{equation}\label{strict}
 		\lim\limits_{\varepsilon\rightarrow 0}\|u^{\varepsilon}\|_{L^{p}\left(\mathbb{R}^{d}\right)}>\|u^{0}\|_{L^{p}\left(\mathbb{R}^{d}\right)}.
 	\end{equation}	
 	We now claim that 
 		 \begin{equation}\label{28}
 		\begin{aligned}	
 			&\liminf\limits_{\varepsilon\rightarrow 0}\iintrdd K(x,y)\Phi'(\uep(y)-\uep(x))(\uep(y)-\uep(x))\Lambda^{\varepsilon}(x,y)\rmd y\rmd x\\
 			\geq&\iintrdd K(x,y)\Phi'(u^0(y)-u^0(x))(u^0(y)-u^0(x))\overline{\Lambda}(x,y)\rmd y\rmd x.
 		\end{aligned}	
 	\end{equation}
 	In fact, we utilize a similar argument to the one in the proof of \eqref{110}.  Since $u^{0}\in W^{\alpha/p,p}\left(\mathbb{R}^{d}\right)$, for any $\kappa>0$, there exists $\delta>0$
 	such that 
 	\begin{equation*}
 		\int_{G_{2}^{\delta}\cup G_{3}^{\delta}}K(x,y)\left|u^{0}(y)-u^{0}(x)\right|^{p}\overline{\Lambda}(x,y)\rmd y\rmd x\leq \kappa.
 	\end{equation*}
 	 Hence, 
 	 \begin{equation}\label{290}
 		\begin{aligned}	
 			&\liminf\limits_{\varepsilon\rightarrow 0}\int_{G_{2}^{\delta}\cup G_{3}^{\delta}}K(x,y)\Phi'(\uep(y)-\uep(x))(\uep(y)-\uep(x))\Lambda^{\varepsilon}(x,y)\rmd y\rmd x\\
 			&\geq 0\geq \int_{G_{2}^{\delta}\cup G_{3}^{\delta}}K(x,y)\Phi'(u^0(y)-u^0(x))(u^0(y)-u^0(x))\overline{\Lambda}(x,y)\rmd y\rmd x-\kappa.
 		\end{aligned}	
 	\end{equation}
 	We proceed to prove that as $\varepsilon\rightarrow 0$,
 	\begin{equation}\label{30}
 		\begin{aligned} &\int_{G_{1}^{\delta}}K(x,y)\Phi'(\uep(y)-\uep(x))(\uep(y)-\uep(x))\Lambda^{\varepsilon}(x,y)\rmd y\rmd x\\
 \rightarrow &\int_{G_{1}^{\delta}}K(x,y)\Phi'(u^0(y)-u^0(x))(u^0(y)-u^0(x))\overline{\Lambda}(x,y)\rmd y\rmd x.
 	\end{aligned}
 \end{equation}
 	Once we have proved \eqref{30}, combining (\ref{290}) with (\ref{30}) yields that
 	\begin{align*}
 		&\liminf\limits_{\varepsilon\rightarrow 0}\iintrdd K(x,y)\Phi'(\uep(y)-\uep(x))(\uep(y)-\uep(x))\Lambda^{\varepsilon}(x,y)\rmd y\rmd x\\
 		&\geq\iintrdd K(x,y)\Phi'(u^0(y)-u^0(x))(u^0(y)-u^0(x))\overline{\Lambda}(x,y)\rmd y\rmd x-\kappa.
 	\end{align*}
 	By the arbitrariness of $\kappa>0$, we obtain the desired relation \eqref{28}.
 	
 	We now prove \eqref{30}. Since $C_{c}^{\infty}\left(\mathbb{R}^{d}\right)$ is dense in $W^{\alpha/p,p}\left(\mathbb{R}^{d}\right)$, for any $\eta>0$, there exists $u_{\eta}\in C_{c}^{\infty}\left(\mathbb{R}^{d}\right)$ such that
 	\begin{align}\label{31}
 		\|u^{0}-u_{\eta}\|_{W^{\alpha/p,p}\left(\mathbb{R}^{d}\right)}<\eta.
 	\end{align}
 	By the symmetry of $K(x,y)$, $\Lambda^{\varepsilon}(x,y)$, and the domain $G_{1}^{\delta}$, we have
 		 \begin{equation}\label{29}
 		\begin{aligned}	
 				&\frac{1}{2}\int_{G_{1}^{\delta}}K(x,y)\Phi'(\uep(y)-\uep(x))(\uep(y)-\uep(x))
 \Lambda^{\varepsilon}(x,y)\rmd y\rmd x\\
 			&=-\int_{G_{1}^{\delta}}K(x,y)\Phi'(\uep(y)-\uep(x))u^{\varepsilon}(x)
 \Lambda^{\varepsilon}(x,y)\rmd y\rmd x\\
 			&=-\int_{G_{1}^{\delta}}K(x,y)\Phi'(\uep(y)-\uep(x))
 \left(u^{\varepsilon}(x)-u^{0}(x)+u^{0}(x)-u_{\eta}(x)+u_{\eta}(x)\right)
 \Lambda^{\varepsilon}(x,y)\rmd y\rmd x\\
 			&=\frac{1}{2}\int_{G_{1}^{\delta}}K(x,y)\Phi'(\uep(y)-\uep(x))
 \left[\left(u^{\varepsilon}-u^{0}\right)(y)-\left(u^{\varepsilon}-u^{0}\right)(x)\right]
 \Lambda^{\varepsilon}(x,y)\rmd y\rmd x\\
 			&\quad+\frac{1}{2}\int_{G_{1}^{\delta}}K(x,y)\Phi'(\uep(y)-\uep(x))
 \left[\left(u^{0}-u_{\eta}\right)(y)-\left(u^{0}-
 u_{\eta}\right)(x)\right]\Lambda^{\varepsilon}(x,y)\rmd y\rmd x\\
 			&\quad+\frac{1}{2}\int_{G_{1}^{\delta}}K(x,y)\Phi'(\uep(y)-\uep(x))
 \left[u_{\eta}(y)-u_{\eta}(x)\right]\Lambda^{\varepsilon}(x,y)\rmd y\rmd x\\
 			&=:\Upsilon_{1}+\Upsilon_{2}+\Upsilon_{3}.
 		\end{aligned}	
 	\end{equation}
 	We consider the integrals $\Upsilon_{1}$, $\Upsilon_{2}$, and $\Upsilon_{3}$ separately. 
 	By \eqref{eq:s3}, (\ref{1}), \eqref{eq:Phi derivation condition},
 \eqref{10}, and H\"{o}lder's inequality, we get
 	\begin{align*}
 		|\Upsilon_{1}|&\lesssim [u^{\varepsilon}]_{W^{s,p}\left(\mathbb{R}^{d}\right)}^{p-1}\left(\int_{G_{1}^{\delta}}
 \frac{\left|\left(u^{\varepsilon}-u^{0}\right)(y)-\left(u^{\varepsilon}
 -u^{0}\right)(x)\right|^{p}}{|x-y|^{d+\alpha}}\rmd y\rmd x\right)^{1/p}\\
 		&\lesssim \left(\int_{\{|y|\leq \delta^{-1}\}}\left|u^{\varepsilon}(y)-u^{0}(y)\right|^{p}\rmd y+\int_{\{|x|\leq \delta^{-1}\}}\left|u^{\varepsilon}(x)-u^{0}(x)\right|^{p}\rmd x\right)^{1/p},
 	\end{align*}
 	which along with the fact that $u^{\varepsilon} \rightarrow u^{0}$ strongly in $\prdloc$ yields that
 	\begin{equation}\label{33}
 		|\Upsilon_{1}|\rightarrow 0 \quad {\rm as \;} \varepsilon\rightarrow 0.
 	\end{equation}
 	Similarly, we have 
 	\begin{align}\label{34}
 		|\Upsilon_{2}|\lesssim [u^{\varepsilon}]_{W^{\alpha/p,p}\left(\mathbb{R}^{d}\right)}^{p-1}[u^{0}-u_{\eta}]_{W^{\alpha/p,p}
 \left(\mathbb{R}^{d}\right)}\lesssim\eta.
 	\end{align}
 	Since $u_{\eta}\in C_{c}^{\infty}\left(\mathbb{R}^{d}\right)$, we directly deduce from \eqref{15} that
 	\begin{align}\label{35}
 		\lim_{\varepsilon\rightarrow 0}\Upsilon_{3}=\frac{1}{2}\int_{G_{1}^{\delta}}K(x,y)\Phi'\left(u^{0}(y)-u^{0}(x)\right)
 \left(u_{\eta}(y)-u_{\eta}(x)\right)\overline{\Lambda}(x,y)\rmd y\rmd x.
 	\end{align}
 	Applying \eqref{eq:Phi derivation condition} and (\ref{31})--(\ref{35}), together with H\"{o}lder's inequality, we arrive at
 	\begin{equation*}
 \begin{aligned}
 &\limsup_{\varepsilon\rightarrow0}\frac{1}{2}
 	  \left\lvert\int_{G_{1}^{\delta}}K(x,y)\Phi'(\uep(y)-\uep(x))(\uep(y)-\uep(x))
 \Lambda^{\varepsilon}(x,y)\rmd y\rmd x\right.\\
&\left.\quad-\int_{G_{1}^{\delta}}K(x,y)\Phi'(u^0(y)-u^0(x))(u^0(y)-u^0(x))\overline{\Lambda}(x,y)\rmd y\rmd x\right\rvert\\
&\leq C\eta+\frac{1}{2}\left|\int_{G_{1}^{\delta}}K(x,y)\Phi'\left(u^{0}(y)-u^{0}(x)\right)
 \left[\left(u_{\eta}-u^{0}\right)(y)-\left(u_{\eta}-u^{0}\right)(x)\right]\overline{\Lambda}(x,y)\rmd y\rmd x\right|\\
 		&\lesssim \eta+[u^{0}]_{W^{\alpha/p,p}\left(\mathbb{R}^{d}\right)}^{p-1}[u_{\eta}-u^{0}]_{W^{\alpha/p,p}
 \left(\mathbb{R}^{d}\right)}\lesssim\eta.
 	\end{aligned}
 	\end{equation*}
 Due to the arbitrariness of $\eta>0$, we get the desired relation \eqref{30}.

 On the one hand, testing equation \eqref{52} against $u^{\varepsilon}$ and passing to the limit,  from the fact that $u^{\varepsilon}$ converges to $u^{0}$ weakly in $L^{p}\left(\mathbb{R}^{d}\right)$ and that $f\in L^{p'}\left(\mathbb{R}^{d}\right)$,
 it follows that
 \begin{equation}\label{posit}
   \begin{aligned}
 		&\lim_{\varepsilon\rightarrow 0}\left(\frac{1}{2}\iintrdd K(x,y)\Phi'(\uep(y)-\uep(x))(\uep(y)-\uep(x))\Lambda^{\varepsilon}(x,y)\rmd y\rmd x+m\int_{\mathbb R^d}|u^{\varepsilon}|^{p}\rmd x\right)\\
 &=\lim_{\varepsilon\rightarrow 0}\int_{\mathbb R^d}fu^{\varepsilon}\rmd x=\int_{\mathbb R^d}fu^{0}\rmd x.
 	\end{aligned}
 \end{equation}
 	On the other hand,
 	Combining \eqref{strict} with (\ref{28}) for the same subsequence we obtain
 	 \begin{equation*}
 		\begin{aligned}	
 			&\lim_{\varepsilon\rightarrow 0}\left(\frac{1}{2}\iintrdd K(x,y)\Phi'(\uep(y)-\uep(x))(\uep(y)-\uep(x))\Lambda^{\varepsilon}(x,y)\rmd y\rmd x+m\int_{\mathbb R^d}|u^{\varepsilon}|^{p}\rmd x\right)\\
 			&>\frac{1}{2}\iintrdd K(x,y)\Phi'(u^0(y)-u^0(x))(u^0(y)-u^0(x))\overline{\Lambda}(x,y)\rmd y\rmd x+m\int_{\mathbb R^d}|u^{0}|^{p}\rmd x
 			=\int_{\mathbb R^d}fu^{0}\rmd x,
 		\end{aligned}	
 	\end{equation*}
 	where we have used the fact that $u^{0}$ satisfies $-L^{0}_{\Phi}u+m|u|^{p-2}u=f$ in the last equality.
 	This contradicts \eqref{posit}. Thus, we conclude that 
 	$u^{\varepsilon} \rightarrow u^{0}$ strongly in $L^{p}\left(\mathbb{R}^{d}\right)$ as $\varepsilon \rightarrow 0 $, which completes the proof of Theorem \ref{thm1}.
 \end{proof}
\section*{Acknowledgments} 
The work of the third author is partially supported by National Natural Science Foundation of China (NSFC Grant No. 12101452) and Young Scientific and Technological Talents (Level Three) in Tianjin, and the work of the corresponding author is partially supported by National Natural Science Foundation of China (NSFC Grant No. 12571103).

\bibliography{ref2}

@article {MR348255,
    AUTHOR = {De Giorgi, E. and Spagnolo, S.},
     TITLE = {Sulla convergenza degli integrali dell'energia per operatori
              ellittici del secondo ordine},
   JOURNAL = {Boll. Un. Mat. Ital. (4)},
  FJOURNAL = {Boll. Un. Mat. Ital. (4)},
    VOLUME = {8},
      YEAR = {1973},
     PAGES = {391--411},
   MRCLASS = {35J15 (47F05)},
  MRNUMBER = {348255},
MRREVIEWER = {M.\ Schechter},
}

@article {MR225015,
    AUTHOR = {Spagnolo, Sergio},
     TITLE = {Sul limite delle soluzioni di problemi di {C}auchy relativi
              all'equazione del calore},
   JOURNAL = {Ann. Scuola Norm. Sup. Pisa Cl. Sci. (3)},
  FJOURNAL = {Annali della Scuola Normale Superiore di Pisa. Classe di
              Scienze. Serie III},
    VOLUME = {21},
      YEAR = {1967},
     PAGES = {657--699},
      ISSN = {0391-173X},
   MRCLASS = {35.63},
  MRNUMBER = {225015},
MRREVIEWER = {R.\ Guenther},
}

@article {MR240443,
    AUTHOR = {Spagnolo, S.},
     TITLE = {Sulla convergenza di soluzioni di equazioni paraboliche ed
              ellittiche},
   JOURNAL = {Ann. Scuola Norm. Sup. Pisa (3)},
  FJOURNAL = {Ann. Scuola Norm. Sup. Pisa (3)},
    VOLUME = {22},
      YEAR = {1968},
     PAGES = {571--597; errata: 22 (1968), 673},
   MRCLASS = {35.42},
  MRNUMBER = {240443},
MRREVIEWER = {L.\ von Wolfersdorf},
}

@book {MR3838419,
    AUTHOR = {Shen, Zhongwei},
     TITLE = {Periodic homogenization of elliptic systems},
    SERIES = {Operator Theory: Advances and Applications},
    VOLUME = {269},
      NOTE = {Advances in Partial Differential Equations (Basel)},
 PUBLISHER = {Birkh\"auser/Springer, Cham},
      YEAR = {2018},
     PAGES = {ix+291},
      ISBN = {978-3-319-91213-4},
   MRCLASS = {35-02 (35B27 35J57 74Q05)},
  MRNUMBER = {3838419},
MRREVIEWER = {Isabelle\ Gruais},
       DOI = {10.1007/978-3-319-91214-1},
       URL = {https://doi.org/10.1007/978-3-319-91214-1},
}

@article {MR542557,
    AUTHOR = {Kozlov, S. M.},
     TITLE = {The averaging of random operators},
   JOURNAL = {Mat. Sb. (N.S.)},
  FJOURNAL = {Matematicheski\u i\ Sbornik. Novaya Seriya},
    VOLUME = {109(151)},
      YEAR = {1979},
    NUMBER = {2},
     PAGES = {188--202, 327},
      ISSN = {0368-8666},
   MRCLASS = {35R60 (60H15)},
  MRNUMBER = {542557},
MRREVIEWER = {Z.\ Schuss},
}

@incollection {MR712714,
    AUTHOR = {Papanicolaou, G. C. and Varadhan, S. R. S.},
     TITLE = {Boundary value problems with rapidly oscillating random
              coefficients},
 BOOKTITLE = {Random fields, {V}ol. {I}, {II} ({E}sztergom, 1979)},
    SERIES = {Colloq. Math. Soc. J\'anos Bolyai},
    VOLUME = {27},
     PAGES = {835--873},
 PUBLISHER = {North-Holland, Amsterdam-New York},
      YEAR = {1981},
      ISBN = {0-444-85441-X},
   MRCLASS = {58G32 (35R60 60H10 82A05)},
  MRNUMBER = {712714},
MRREVIEWER = {Jacques\ Vauthier},
}

@article {MR4000213,
    AUTHOR = {Kassmann, M. and Piatnitski, A. and Zhizhina, E.},
     TITLE = {Homogenization of {L}\'evy-type operators with oscillating
              coefficients},
   JOURNAL = {SIAM J. Math. Anal.},
  FJOURNAL = {SIAM Journal on Mathematical Analysis},
    VOLUME = {51},
      YEAR = {2019},
    NUMBER = {5},
     PAGES = {3641--3665},
      ISSN = {0036-1410,1095-7154},
   MRCLASS = {45E10 (35B27 45M05 60J75)},
  MRNUMBER = {4000213},
MRREVIEWER = {Andreas\ Rathsfeld},
       DOI = {10.1137/18M1200038},
       URL = {https://doi.org/10.1137/18M1200038},
}

@article {MR4111815,
    AUTHOR = {Chaker, Jamil and Silvestre, Luis},
     TITLE = {Coercivity estimates for integro-differential operators},
   JOURNAL = {Calc. Var. Partial Differential Equations},
  FJOURNAL = {Calculus of Variations and Partial Differential Equations},
    VOLUME = {59},
      YEAR = {2020},
    NUMBER = {4},
     PAGES = {Paper No. 106, 20},
      ISSN = {0944-2669,1432-0835},
   MRCLASS = {47G20 (26A33 35R09)},
  MRNUMBER = {4111815},
MRREVIEWER = {David\ Kapanadze},
       DOI = {10.1007/s00526-020-01764-y},
       URL = {https://doi.org/10.1007/s00526-020-01764-y},
}

@book {MR4769823,
    AUTHOR = {Fern\'andez-Real, Xavier and Ros-Oton, Xavier},
     TITLE = {Integro-differential elliptic equations},
    SERIES = {Progress in Mathematics},
    VOLUME = {350},
 PUBLISHER = {Birkh\"auser/Springer, Cham},
      YEAR = {2024},
     PAGES = {xvi+395},
      ISBN = {978-3-031-54241-1; 978-3-031-54242-8},
   MRCLASS = {45-02 (35B65 35J60 35J86 35R09 35R35)},
  MRNUMBER = {4769823},
       DOI = {10.1007/978-3-031-54242-8},
       URL = {https://doi.org/10.1007/978-3-031-54242-8},
}

@article {MR4049224,
    AUTHOR = {Imbert, Cyril and Silvestre, Luis},
     TITLE = {The weak {H}arnack inequality for the {B}oltzmann equation
              without cut-off},
   JOURNAL = {J. Eur. Math. Soc.},
  FJOURNAL = {Journal of the European Mathematical Society (JEMS)},
    VOLUME = {22},
      YEAR = {2020},
    NUMBER = {2},
     PAGES = {507--592},
      ISSN = {1435-9855,1435-9863},
   MRCLASS = {35Q20 (35B45 35B65 35R09 82C40)},
  MRNUMBER = {4049224},
MRREVIEWER = {Andrei\ Tarfulea},
       DOI = {10.4171/jems/928},
       URL = {https://doi.org/10.4171/jems/928},
}

@book {MR1201152,
    AUTHOR = {Dal Maso, Gianni},
     TITLE = {An introduction to {$\Gamma$}-convergence},
    SERIES = {Progress in Nonlinear Differential Equations and their
              Applications},
    VOLUME = {8},
 PUBLISHER = {Birkh\"auser Boston, Inc., Boston, MA},
      YEAR = {1993},
     PAGES = {xiv+340},
      ISBN = {0-8176-3679-X},
   MRCLASS = {49-02 (46N10 47H99 47N10 49J45 73B27)},
  MRNUMBER = {1201152},
MRREVIEWER = {T.\ Zolezzi},
       DOI = {10.1007/978-1-4612-0327-8},
       URL = {https://doi.org/10.1007/978-1-4612-0327-8},
}

@article {MR375037,
    AUTHOR = {De Giorgi, Ennio},
     TITLE = {Sulla convergenza di alcune successioni d'integrali del tipo
              dell'area},
   JOURNAL = {Rend. Mat. (6)},
  FJOURNAL = {Rendiconti di Matematica. Serie VI.},
    VOLUME = {8},
      YEAR = {1975},
     PAGES = {277--294},
      ISSN = {0034-4427},
   MRCLASS = {49A50 (35J20)},
  MRNUMBER = {375037},
MRREVIEWER = {L.\ von Wolfersdorf},
}

@article {MR448194,
    AUTHOR = {De Giorgi, Ennio and Franzoni, Tullio},
     TITLE = {Su un tipo di convergenza variazionale},
   JOURNAL = {Atti Accad. Naz. Lincei Rend. Cl. Sci. Fis. Mat. Nat. (8)},
  FJOURNAL = {Atti della Accademia Nazionale dei Lincei. Rendiconti. Classe
              di Scienze Fisiche, Matematiche e Naturali. Serie VIII},
    VOLUME = {58},
      YEAR = {1975},
    NUMBER = {6},
     PAGES = {842--850},
      ISSN = {0392-7881},
   MRCLASS = {49A99},
  MRNUMBER = {448194},
MRREVIEWER = {S.\ Cinquini},
}

@book {MR1968440,
    AUTHOR = {Braides, Andrea},
     TITLE = {{$\Gamma$}-convergence for beginners},
    SERIES = {Oxford Lecture Series in Mathematics and its Applications},
    VOLUME = {22},
 PUBLISHER = {Oxford University Press, Oxford},
      YEAR = {2002},
     PAGES = {xii+218},
      ISBN = {0-19-850784-4},
   MRCLASS = {49-02 (35B27 49J45 49K40 74Q05)},
  MRNUMBER = {1968440},
MRREVIEWER = {Ilaria\ Fragal{\`a}},
       DOI = {10.1093/acprof:oso/9780198507840.001.0001},
       URL = {https://doi.org/10.1093/acprof:oso/9780198507840.001.0001},
}

@book {MR2759829,
    AUTHOR = {Brezis, Haim},
     TITLE = {Functional analysis, {S}obolev spaces and partial differential
              equations},
    SERIES = {Universitext},
 PUBLISHER = {Springer, New York},
      YEAR = {2011},
     PAGES = {xiv+599},
      ISBN = {978-0-387-70913-0},
   MRCLASS = {35-01 (46-01 46E35 46N20 47F05)},
  MRNUMBER = {2759829},
MRREVIEWER = {Vicen\c tiu\ D.\ R\u adulescu},
}

@book {MR1727362,
    AUTHOR = {Ekeland, Ivar and T{\'e}mam, Roger},
     TITLE = {Convex analysis and variational problems},
    SERIES = {Classics in Applied Mathematics},
    VOLUME = {28},
   EDITION = {English},
      NOTE = {Translated from the French},
 PUBLISHER = {Society for Industrial and Applied Mathematics (SIAM),
              Philadelphia, PA},
      YEAR = {1999},
     PAGES = {xiv+402},
      ISBN = {0-89871-450-8},
   MRCLASS = {49-02 (01A75 49J53 90C46)},
  MRNUMBER = {1727362},
       DOI = {10.1137/1.9781611971088},
       URL = {https://doi.org/10.1137/1.9781611971088},
}

@book{jikov2012homogenization,
  title={Homogenization of differential operators and integral functionals},
  author={Jikov, Vasili Vasilievitch and Kozlov, Sergei M and Oleinik, Olga Arsenievna},
  year={2012},
  publisher={Springer Science \& Business Media}
}

@article {MR2480109,
    AUTHOR = {Gilboa, Guy and Osher, Stanley},
     TITLE = {Nonlocal operators with applications to image processing},
   JOURNAL = {Multiscale Model. Simul.},
  FJOURNAL = {Multiscale Modeling \& Simulation. A SIAM Interdisciplinary
              Journal},
    VOLUME = {7},
      YEAR = {2008},
    NUMBER = {3},
     PAGES = {1005--1028},
      ISSN = {1540-3459,1540-3467},
   MRCLASS = {94A08 (35A15 68U10)},
  MRNUMBER = {2480109},
MRREVIEWER = {Bartomeu\ Coll},
       DOI = {10.1137/070698592},
       URL = {https://doi.org/10.1137/070698592},
}

@article {MR2120547,
    AUTHOR = {Carrillo, C. and Fife, P.},
     TITLE = {Spatial effects in discrete generation population models},
   JOURNAL = {J. Math. Biol.},
  FJOURNAL = {Journal of Mathematical Biology},
    VOLUME = {50},
      YEAR = {2005},
    NUMBER = {2},
     PAGES = {161--188},
      ISSN = {0303-6812,1432-1416},
   MRCLASS = {92D25 (39A11)},
  MRNUMBER = {2120547},
MRREVIEWER = {Carol\ B.\ Overdeep},
       DOI = {10.1007/s00285-004-0284-4},
       URL = {https://doi.org/10.1007/s00285-004-0284-4},
}

@article {MR3327993,
    AUTHOR = {Tanzy, M. C. and Volpert, V. A. and Bayliss, A. and Nehrkorn,
              M. E.},
     TITLE = {A {N}agumo-type model for competing populations with nonlocal
              coupling},
   JOURNAL = {Math. Biosci.},
  FJOURNAL = {Mathematical Biosciences},
    VOLUME = {263},
      YEAR = {2015},
     PAGES = {70--82},
      ISSN = {0025-5564,1879-3134},
   MRCLASS = {92D25},
  MRNUMBER = {3327993},
       DOI = {10.1016/j.mbs.2015.01.014},
       URL = {https://doi.org/10.1016/j.mbs.2015.01.014},
}

@article {MR3595876,
    AUTHOR = {Piatnitski, A. and Zhizhina, E.},
     TITLE = {Periodic homogenization of nonlocal operators with a
              convolution-type kernel},
   JOURNAL = {SIAM J. Math. Anal.},
  FJOURNAL = {SIAM Journal on Mathematical Analysis},
    VOLUME = {49},
      YEAR = {2017},
    NUMBER = {1},
     PAGES = {64--81},
      ISSN = {0036-1410,1095-7154},
   MRCLASS = {45E10 (35B10 35B27 45M05 60J75)},
  MRNUMBER = {3595876},
MRREVIEWER = {Andreas\ Rathsfeld},
       DOI = {10.1137/16M1072292},
       URL = {https://doi.org/10.1137/16M1072292},
}

@article {MR4053030,
    AUTHOR = {Piatnitski, A. and Zhizhina, E.},
     TITLE = {Stochastic homogenization of convolution type operators},
   JOURNAL = {J. Math. Pures Appl. (9)},
  FJOURNAL = {Journal de Math{\'e}matiques Pures et Appliqu{\'e}es. Neuvi{\`e}me
              S{\'}erie},
    VOLUME = {134},
      YEAR = {2020},
     PAGES = {36--71},
      ISSN = {0021-7824,1776-3371},
   MRCLASS = {35B27 (45E10 47B25 60H25)},
  MRNUMBER = {4053030},
       DOI = {10.1016/j.matpur.2019.12.001},
       URL = {https://doi.org/10.1016/j.matpur.2019.12.001},
}

@article {MR4033742,
    AUTHOR = {Piatnitski, A. and Zhizhina, E.},
     TITLE = {Homogenization of biased convolution type operators},
   JOURNAL = {Asymptot. Anal.},
  FJOURNAL = {Asymptotic Analysis},
    VOLUME = {115},
      YEAR = {2019},
    NUMBER = {3-4},
     PAGES = {241--262},
      ISSN = {0921-7134,1875-8576},
   MRCLASS = {35B27 (35R09)},
  MRNUMBER = {4033742},
       DOI = {10.3233/asy-191533},
       URL = {https://doi.org/10.3233/asy-191533},
}

@article {MR4853448,
    AUTHOR = {Piatnitski, A. and Zhizhina, E.},
     TITLE = {Homogenization of non-autonomous evolution problems for
              convolution type operators in randomly evolving media},
   JOURNAL = {J. Math. Pures Appl. (9)},
  FJOURNAL = {Journal de Math\'ematiques Pures et Appliqu\'ees. Neuvi\`eme
              S\'erie},
    VOLUME = {194},
      YEAR = {2025},
     PAGES = {Paper No. 103660, 30},
      ISSN = {0021-7824,1776-3371},
   MRCLASS = {35B27 (35K90 37A25 45E10 45M05)},
  MRNUMBER = {4853448},
MRREVIEWER = {Longjuan\ Xu},
       DOI = {10.1016/j.matpur.2025.103660},
       URL = {https://doi.org/10.1016/j.matpur.2025.103660},
}

@article {MR4532955,
    AUTHOR = {Piatnitski, A. and Sloushch, V. and Suslina, T. and Zhizhina,
              E.},
     TITLE = {On operator estimates in homogenization of nonlocal operators
              of convolution type},
   JOURNAL = {J. Differential Equations},
  FJOURNAL = {Journal of Differential Equations},
    VOLUME = {352},
      YEAR = {2023},
     PAGES = {153--188},
      ISSN = {0022-0396,1090-2732},
   MRCLASS = {47G10 (47N20)},
  MRNUMBER = {4532955},
       DOI = {10.1016/j.jde.2022.12.036},
       URL = {https://doi.org/10.1016/j.jde.2022.12.036},
}

@article {MR4720001,
    AUTHOR = {Piatnitski, A. and Sloushch, V. and Suslina, T. and Zhizhina,
              E.},
     TITLE = {On the homogenization of nonlocal convolution type operators},
   JOURNAL = {Russ. J. Math. Phys.},
  FJOURNAL = {Russian Journal of Mathematical Physics},
    VOLUME = {31},
      YEAR = {2024},
    NUMBER = {1},
     PAGES = {137--145},
      ISSN = {1061-9208,1555-6638},
   MRCLASS = {47G10 (47F10)},
  MRNUMBER = {4720001},
       DOI = {10.1134/S106192084010114},
       URL = {https://doi.org/10.1134/S106192084010114},
}

@article {MR1867081,
    AUTHOR = {Nualart, David and Schoutens, Wim},
     TITLE = {Backward stochastic differential equations and {F}eynman-{K}ac
              formula for {L}{\'e}vy processes, with applications in finance},
   JOURNAL = {Bernoulli},
  FJOURNAL = {Bernoulli. Official Journal of the Bernoulli Society for
              Mathematical Statistics and Probability},
    VOLUME = {7},
      YEAR = {2001},
    NUMBER = {5},
     PAGES = {761--776},
      ISSN = {1350-7265,1573-9759},
   MRCLASS = {60H10 (62P05 91B28)},
  MRNUMBER = {1867081},
       DOI = {10.2307/3318541},
       URL = {https://doi.org/10.2307/3318541},
}

@article {MR2560294,
    AUTHOR = {Arisawa, Mariko},
     TITLE = {Homogenization of a class of integro-differential equations
              with {L}\'evy operators},
   JOURNAL = {Comm. Partial Differential Equations},
  FJOURNAL = {Communications in Partial Differential Equations},
    VOLUME = {34},
      YEAR = {2009},
    NUMBER = {7-9},
     PAGES = {617--624},
      ISSN = {0360-5302,1532-4133},
   MRCLASS = {35B27 (45K05 47G20 74Q10)},
  MRNUMBER = {2560294},
MRREVIEWER = {Karsten\ Matthies},
       DOI = {10.1080/03605300902963518},
       URL = {https://doi.org/10.1080/03605300902963518},
}

@article {MR2981018,
    AUTHOR = {Arisawa, Mariko},
     TITLE = {Homogenizations of integro-differential equations with
              {L}\'evy operators with asymmetric and degenerate densities},
   JOURNAL = {Proc. Roy. Soc. Edinburgh Sect. A},
  FJOURNAL = {Proceedings of the Royal Society of Edinburgh. Section A.
              Mathematics},
    VOLUME = {142},
      YEAR = {2012},
    NUMBER = {5},
     PAGES = {917--943},
      ISSN = {0308-2105,1473-7124},
   MRCLASS = {35B27 (35R09 45K05 47G20)},
  MRNUMBER = {2981018},
       DOI = {10.1017/S0308210510001897},
       URL = {https://doi.org/10.1017/S0308210510001897},
}

@article {MR4793281,
    AUTHOR = {Fern\'andez-Real, Xavier and Ros-Oton, Xavier},
     TITLE = {Schauder and {C}ordes-{N}irenberg estimates for nonlocal
              elliptic equations with singular kernels},
   JOURNAL = {Proc. Lond. Math. Soc. (3)},
  FJOURNAL = {Proceedings of the London Mathematical Society. Third Series},
    VOLUME = {129},
      YEAR = {2024},
    NUMBER = {3},
     PAGES = {Paper No. e12629, 47},
      ISSN = {0024-6115,1460-244X},
   MRCLASS = {47G20 (35B65 35R09 35R11 35S05 60G52)},
  MRNUMBER = {4793281},
MRREVIEWER = {Nino\ Manjavidze},
       DOI = {10.1112/plms.12629},
       URL = {https://doi.org/10.1112/plms.12629},
}

@article {MR2085184,
	AUTHOR = {Zhikov, V. V.},
	TITLE = {On two-scale convergence},
	JOURNAL = {Tr. Semin. im. I. G. Petrovskogo},
	FJOURNAL = {Trudy Seminara imeni I. G. Petrovskogo},
	YEAR = {2003},
	NUMBER = {23},
	PAGES = {149--187, 410},
	ISSN = {0321-2971},
	MRCLASS = {35B27 (28A99 49J45 74Q05)},
	MRNUMBER = {2085184},
	DOI = {10.1023/B:JOTH.0000016052.48558.b4},
	URL = {https://doi.org/10.1023/B:JOTH.0000016052.48558.b4},
}

@book {MR3931688,
	AUTHOR = {Lindqvist, Peter},
	TITLE = {Notes on the stationary {$p$}-{L}aplace equation},
	SERIES = {SpringerBriefs in Mathematics},
	PUBLISHER = {Springer, Cham},
	YEAR = {2019},
	PAGES = {xi+104},
	ISBN = {978-3-030-14500-2; 978-3-030-14501-9},
	MRCLASS = {35-02 (31C45 35A01 35B65 35D30 35D40 35J60 35J92)},
	MRNUMBER = {3931688},
	MRREVIEWER = {Vladimir\ Bobkov},
	DOI = {10.1007/978-3-030-14501-9},
	URL = {https://doi.org/10.1007/978-3-030-14501-9},
}

@article {MR847306,
	AUTHOR = {Giaquinta, Mariano and Modica, Giuseppe},
	TITLE = {Partial regularity of minimizers of quasiconvex integrals},
	JOURNAL = {Ann. Inst. H. Poincar\'e{} Anal. Non Lin\'eaire},
	FJOURNAL = {Annales de l'Institut Henri Poincar\'e. Analyse Non
	Lin\'eaire},
	VOLUME = {3},
	YEAR = {1986},
	NUMBER = {3},
	PAGES = {185--208},
	ISSN = {0294-1449},
	MRCLASS = {49B36 (49B21)},
	MRNUMBER = {847306},
	MRREVIEWER = {T.\ Zolezzi},
	URL = {http://www.numdam.org/item?id=AIHPC_1986__3_3_185_0},
}

@article {MR4919132,
	AUTHOR = {B\"ogelein, Verena and Duzaar, Frank and Liao, Naian and
	Molica Bisci, Giovanni and Servadei, Raffaella},
	TITLE = {Regularity for the fractional {$p$}-{L}aplace equation},
	JOURNAL = {J. Funct. Anal.},
	FJOURNAL = {Journal of Functional Analysis},
	VOLUME = {289},
	YEAR = {2025},
	NUMBER = {9},
	PAGES = {Paper No. 111078},
	ISSN = {0022-1236,1096-0783},
	MRCLASS = {35B65 (35J70 35R09 47G20)},
	MRNUMBER = {4919132},
	DOI = {10.1016/j.jfa.2025.111078},
	URL = {https://doi.org/10.1016/j.jfa.2025.111078},
}

@book {MR4567945,
	AUTHOR = {Leoni, Giovanni},
	TITLE = {A first course in fractional {S}obolev spaces},
	SERIES = {Graduate Studies in Mathematics},
	VOLUME = {229},
	PUBLISHER = {American Mathematical Society, Providence, RI},
	YEAR = {2023},
	PAGES = {xv+586},
	ISBN = {[9781470468989]; [9781470472535]; [9781470472528]},
	MRCLASS = {46-01 (30H05 35R11 42Bxx 42C40 46E35)},
	MRNUMBER = {4567945},
	MRREVIEWER = {E.\ S.\ Dubtsov},
	DOI = {10.1090/gsm/229},
	URL = {https://doi.org/10.1090/gsm/229},
}

@article{Jin2025,
  author  = {Xiaofeng Jin and Lingwei Ma and Zhenqiu Zhang},
  title   = {Qualitative Stochastic Homogenization of Elliptic Equations With Random Coefficients and Convolutional Potentials},
  journal = {Mathematical Methods in the Applied Sciences},
  howpublished={Early View},
  year    = {2025},
  doi     = {10.1002/mma.11182},
  note     = {\url{https://doi.org/10.1002/mma.11182}},
}

@article {MR3862089,
    AUTHOR = {Shen, Zhongwei and Zhuge, Jinping},
     TITLE = {Boundary layers in periodic homogenization of {N}eumann
              problems},
   JOURNAL = {Comm. Pure Appl. Math.},
  FJOURNAL = {Communications on Pure and Applied Mathematics},
    VOLUME = {71},
      YEAR = {2018},
    NUMBER = {11},
     PAGES = {2163--2219},
      ISSN = {0010-3640,1097-0312},
   MRCLASS = {35B27},
  MRNUMBER = {3862089},
MRREVIEWER = {Annalisa\ Cesaroni},
       DOI = {10.1002/cpa.21740},
       URL = {https://doi.org/10.1002/cpa.21740},
}

@article {MR3902170,
    AUTHOR = {Zhuge, Jinping},
     TITLE = {Homogenization and boundary layers in domains of finite type},
   JOURNAL = {Comm. Partial Differential Equations},
  FJOURNAL = {Communications in Partial Differential Equations},
    VOLUME = {43},
      YEAR = {2018},
    NUMBER = {4},
     PAGES = {549--584},
      ISSN = {0360-5302,1532-4133},
   MRCLASS = {35B27 (35C20 35J57 74Q05)},
  MRNUMBER = {3902170},
MRREVIEWER = {Marcus\ Waurick},
       DOI = {10.1080/03605302.2018.1446160},
       URL = {https://doi.org/10.1080/03605302.2018.1446160},
}

@article {MR4566686,
    AUTHOR = {Niu, Weisheng and Zhuge, Jinping},
     TITLE = {Compactness and stable regularity in multiscale
              homogenization},
   JOURNAL = {Math. Ann.},
  FJOURNAL = {Mathematische Annalen},
    VOLUME = {385},
      YEAR = {2023},
    NUMBER = {3-4},
     PAGES = {1431--1473},
      ISSN = {0025-5831,1432-1807},
   MRCLASS = {35B27 (35B65)},
  MRNUMBER = {4566686},
MRREVIEWER = {Alain\ Brillard},
       DOI = {10.1007/s00208-022-02378-9},
       URL = {https://doi.org/10.1007/s00208-022-02378-9},
}

@article {MR4840548,
    AUTHOR = {Geng, Jun and Shi, Bojing},
     TITLE = {Quantitative estimates in almost periodic homogenization of
              parabolic systems},
   JOURNAL = {Calc. Var. Partial Differential Equations},
  FJOURNAL = {Calculus of Variations and Partial Differential Equations},
    VOLUME = {64},
      YEAR = {2025},
    NUMBER = {1},
     PAGES = {Paper No. 33, 57},
      ISSN = {0944-2669,1432-0835},
   MRCLASS = {35B27 (35A08 35K40)},
  MRNUMBER = {4840548},
MRREVIEWER = {Ahmed\ Boughammoura},
       DOI = {10.1007/s00526-024-02881-8},
       URL = {https://doi.org/10.1007/s00526-024-02881-8},
}

@article {MR4377865,
    AUTHOR = {Gloria, Antoine and Neukamm, Stefan and Otto, Felix},
     TITLE = {Quantitative estimates in stochastic homogenization for
              correlated coefficient fields},
   JOURNAL = {Anal. PDE},
  FJOURNAL = {Analysis \& PDE},
    VOLUME = {14},
      YEAR = {2021},
    NUMBER = {8},
     PAGES = {2497--2537},
      ISSN = {2157-5045,1948-206X},
   MRCLASS = {35J15 (35J47 60H25 74Q05)},
  MRNUMBER = {4377865},
       DOI = {10.2140/apde.2021.14.2497},
       URL = {https://doi.org/10.2140/apde.2021.14.2497},
}

@book {MR3932093,
    AUTHOR = {Armstrong, Scott and Kuusi, Tuomo and Mourrat,
              Jean-Christophe},
     TITLE = {Quantitative stochastic homogenization and large-scale
              regularity},
    SERIES = {Grundlehren der mathematischen Wissenschaften [Fundamental
              Principles of Mathematical Sciences]},
    VOLUME = {352},
 PUBLISHER = {Springer, Cham},
      YEAR = {2019},
     PAGES = {xxxviii+518},
      ISBN = {978-3-030-15544-5; 978-3-030-15545-2; 978-3-030-15547-6},
   MRCLASS = {35-02 (35B27 60F17 60H15)},
  MRNUMBER = {3932093},
       DOI = {10.1007/978-3-030-15545-2},
       URL = {https://doi.org/10.1007/978-3-030-15545-2},
}

@article {MR4261110,
    AUTHOR = {Chen, Xin and Chen, Zhen-Qing and Kumagai, Takashi and Wang,
              Jian},
     TITLE = {Homogenization of symmetric stable-like processes in
              stationary ergodic media},
   JOURNAL = {SIAM J. Math. Anal.},
  FJOURNAL = {SIAM Journal on Mathematical Analysis},
    VOLUME = {53},
      YEAR = {2021},
    NUMBER = {3},
     PAGES = {2957--3001},
      ISSN = {0036-1410,1095-7154},
   MRCLASS = {60G51 (60G52 60J25 60J76)},
  MRNUMBER = {4261110},
       DOI = {10.1137/20M1326726},
       URL = {https://doi.org/10.1137/20M1326726},
}

@article {MR4103433,
    AUTHOR = {Gloria, Antoine and Neukamm, Stefan and Otto, Felix},
     TITLE = {A regularity theory for random elliptic operators},
   JOURNAL = {Milan J. Math.},
  FJOURNAL = {Milan Journal of Mathematics},
    VOLUME = {88},
      YEAR = {2020},
    NUMBER = {1},
     PAGES = {99--170},
      ISSN = {1424-9286,1424-9294},
   MRCLASS = {35R60 (35B27 35B65 60H25)},
  MRNUMBER = {4103433},
       DOI = {10.1007/s00032-020-00309-4},
       URL = {https://doi.org/10.1007/s00032-020-00309-4},
}

@article {MR3713047,
    AUTHOR = {Gloria, Antoine and Otto, Felix},
     TITLE = {Quantitative results on the corrector equation in stochastic
              homogenization},
   JOURNAL = {J. Eur. Math. Soc.},
  FJOURNAL = {Journal of the European Mathematical Society (JEMS)},
    VOLUME = {19},
      YEAR = {2017},
    NUMBER = {11},
     PAGES = {3489--3548},
      ISSN = {1435-9855,1435-9863},
   MRCLASS = {35B27 (35J25 39A70 60H15)},
  MRNUMBER = {3713047},
       DOI = {10.4171/JEMS/745},
       URL = {https://doi.org/10.4171/JEMS/745},
}

@article {MR4348681,
    AUTHOR = {Chen, Xin and Chen, Zhen-Qing and Kumagai, Takashi and Wang,
              Jian},
     TITLE = {Periodic homogenization of nonsymmetric {L}\'evy-type
              processes},
   JOURNAL = {Ann. Probab.},
  FJOURNAL = {The Annals of Probability},
    VOLUME = {49},
      YEAR = {2021},
    NUMBER = {6},
     PAGES = {2874--2921},
      ISSN = {0091-1798,2168-894X},
   MRCLASS = {60G51 (60F17 60J25 60J35 60J76)},
  MRNUMBER = {4348681},
MRREVIEWER = {Suprio\ Bhar},
       DOI = {10.1214/21-aop1518},
       URL = {https://doi.org/10.1214/21-aop1518},
}

@incollection {MR4132119,
    AUTHOR = {Du, Qiang and Engquist, Bjorn and Tian, Xiaochuan},
     TITLE = {Multiscale modeling, homogenization and nonlocal effects:
              mathematical and computational issues},
 BOOKTITLE = {75 years of mathematics of computation},
    SERIES = {Contemp. Math.},
    VOLUME = {754},
     PAGES = {115--139},
 PUBLISHER = {Amer. Math. Soc., [Providence], RI},
      YEAR = {2020},
      ISBN = {978-1-4704-5163-9},
   MRCLASS = {65-02 (00A71 35B27 65N99 65R20 70-08 74Q99)},
  MRNUMBER = {4132119},
       DOI = {10.1090/conm/754/15175},
       URL = {https://doi.org/10.1090/conm/754/15175},
}

@article {MR3339179,
    AUTHOR = {Kuusi, Tuomo and Mingione, Giuseppe and Sire, Yannick},
     TITLE = {Nonlocal equations with measure data},
   JOURNAL = {Comm. Math. Phys.},
  FJOURNAL = {Communications in Mathematical Physics},
    VOLUME = {337},
      YEAR = {2015},
    NUMBER = {3},
     PAGES = {1317--1368},
      ISSN = {0010-3616,1432-0916},
   MRCLASS = {35R11 (31B35 35A01 35R06 35R09)},
  MRNUMBER = {3339179},
MRREVIEWER = {Paolo\ Baroni},
       DOI = {10.1007/s00220-015-2356-2},
       URL = {https://doi.org/10.1007/s00220-015-2356-2},
}

@article {MR4693935,
    AUTHOR = {De Filippis, Cristiana and Mingione, Giuseppe},
     TITLE = {Gradient regularity in mixed local and nonlocal problems},
   JOURNAL = {Math. Ann.},
  FJOURNAL = {Mathematische Annalen},
    VOLUME = {388},
      YEAR = {2024},
    NUMBER = {1},
     PAGES = {261--328},
      ISSN = {0025-5831,1432-1807},
   MRCLASS = {49N60 (35J60 35R11 49J10)},
  MRNUMBER = {4693935},
MRREVIEWER = {Xiaodong\ Yan},
       DOI = {10.1007/s00208-022-02512-7},
       URL = {https://doi.org/10.1007/s00208-022-02512-7},
}

\end{document}